%% file: Biperiodic-13.tex
\documentclass[11pt]{amsart}

\usepackage{amsfonts, amstext, amsmath, amsthm, amscd, amssymb}
\usepackage{graphicx, color,  subfigure, wrapfig, overpic, url}
\usepackage[dvipsnames]{xcolor}
\usepackage{microtype}
\usepackage{import}
\usepackage{tikz}

\usepackage{marginnote}
\long\def\@savemarbox#1#2{\global\setbox#1\vtop{\hsize\marginparwidth 
  \@parboxrestore\tiny\raggedright #2}}
\AtBeginDocument{
   \def\MR#1{}  }

\newcommand{\R}{\mathbb{R}}
\newcommand{\Q}{\mathbb{Q}}

\newcommand{\CC}{\mathbb{C}}
\newcommand{\RR}{\mathbb{R}}
\newcommand{\HH}{\mathbb{H}}

\newcommand{\QQ}{\mathbb{Q}}

\newcommand{\T}{\mathcal T}
\newcommand{\A}{\mathcal A}
\newcommand{\W}{\mathcal W}
\newcommand{\vol}{{\rm vol}}

\newcommand{\voct}{{v_{\rm oct}}}
\newcommand{\vtet}{{v_{\rm tet}}}
\renewcommand{\L}{\mathcal L}

\newcommand{\toF}{\stackrel{\rm{F}}{\to}}
\renewcommand{\setminus}{{\smallsetminus}}

\newcommand{\from}{\colon\thinspace}
\newcommand{\volp}{{\rm vol}^{\perp}}
\newcommand{\volbp}{{\rm vol}^{\lozenge}}

\numberwithin{equation}{section}
\theoremstyle{plain}
\newtheorem{theorem}{Theorem}[section]
\newtheorem{corollary}[theorem]{Corollary}
\newtheorem{lemma}[theorem]{Lemma}

\newtheorem{conjecture}[theorem]{Conjecture}

\newtheorem*{namedtheorem}{\theoremname}
\newcommand{\theoremname}{testing}

\theoremstyle{definition}
\newtheorem{definition}[theorem]{Definition}
\newtheorem{question}[theorem]{Question}
\newtheorem{remark}[theorem]{Remark}

\newcommand{\refthm}[1]{Theorem~\ref{Thm:#1}}
\newcommand{\reflem}[1]{Lemma~\ref{Lem:#1}}

\newcommand{\refdef}[1]{Definition~\ref{Def:#1}}

\newcommand{\reffig}[1]{Figure~\ref{Fig:#1}}

\title[Geometry of biperiodic alternating links]%
{Geometry of biperiodic alternating links}
\author[A.\ Champanerkar]{Abhijit Champanerkar}
\address{Department of Mathematics, College of Staten Island \& The Graduate Center, City University of New York, New York, NY}
\email{abhijit@math.csi.cuny.edu}
\author[I. \ Kofman]{Ilya Kofman}
\address{Department of Mathematics, College of Staten Island \& The Graduate Center, City University of New York, New York, NY}
\email{ikofman@math.csi.cuny.edu}
\author[J. \ Purcell]{Jessica S.\ Purcell}
\address{School of Mathematical Sciences, 9 Rainforest Walk, Monash University, Victoria 3800, Australia}
\email{jessica.purcell@monash.edu}

\begin{document}

\begin{abstract}
A biperiodic alternating link has an alternating quotient link in the
thickened torus. In this paper, we focus on semi-regular links, a
class of biperiodic alternating links whose hyperbolic structure can
be immediately determined from a corresponding Euclidean
tiling. Consequently, we determine the exact volumes of semi-regular
links. We relate their commensurability and arithmeticity to the
corresponding tiling, and assuming a conjecture of Milnor, we show
there exist infinitely many pairwise incommensurable semi-regular
links with the same invariant trace field. We show that only two
semi-regular links have totally geodesic checkerboard surfaces; these
two links satisfy the Volume Density Conjecture.  Finally, we give
conditions implying that many additional biperiodic alternating links
are hyperbolic and admit a positively oriented, unimodular geometric
triangulation.  We also provide sharp upper and lower volume bounds
for these links.
\end{abstract}

\maketitle

%%%%%%%%%%%%%%%%%%%%%%%%%%%%%%%%%%%%%%%%%%%%%%%%%%%%%%%%%%%%%%%%%
\section{Introduction}

It is well known, due to work of Menasco in the 1980s, that there
exists a hyperbolic structure on the complement of any prime
alternating link in $S^3$ that is not a $(2,q)$-torus link
\cite{menasco}. The alternating diagram also provides a natural
explicit decomposition of the link complement into two ideal
\emph{checkerboard polyhedra} with faces identified
\cite{menasco:polyhedra, ALR:polyhedra}. Aitchison and Reeves
\cite{aitchison_reeves} studied alternating links for which these
combinatorial polyhedra can be realized directly as ideal hyperbolic
polyhedra that can be glued together to obtain the complete hyperbolic
structure on the link complement. They called such links
\emph{completely realizable}, and they described a large
family of completely realizable alternating links, called Archimedean
links.  Because the geometry of a completely realizable link is
explicit, and matches the combinatorics of the polyhedral
decomposition, it is easy to determine various geometric properties
from its diagram, such as its exact hyperbolic volume. 
However, there are only finitely many Archimedean links.

There are also a few known infinite families of completely realizable alternating links in $S^3$, built from prisms and antiprisms; these are mentioned in \cite{aitchison_reeves}, and described in some detail in Thurston's notes \cite[Section~6.8]{thurston:notes}. As far as we are aware, these are the only possibilities for completely realizable alternating links in the 3-sphere. 

In this paper we show that, in contrast, there are many infinite families of alternating links in the thickened torus $T^2\times I$ that are completely realizable, for an analogous notion of
a polyhedral decomposition.  Moreover, the geometric structure of
these links, called semi-regular links, can be immediately determined
from a corresponding Euclidean tiling. Thus geometric invariants, such
as hyperbolic volume, arithmeticity and commensurability, can be
computed directly from this tiling.

When lifted to the universal cover of $T^2\times I$, a link $L$ in
$T^2\times I$ becomes a biperiodic link $\L$ in $\R^3$. Conversely,
any biperiodic alternating link $\L$ is invariant under translations
by a two-dimensional lattice $\Lambda$, such that $L=\L/\Lambda$ is an
alternating link in $T^2\times I$. Thus, equivalently, we prove that
infinitely many biperiodic alternating links are completely
realizable, and we determine their geometric properties. Note that the
hyperbolic structures of these completely realizable biperiodic
alternating links were discovered independently by Adams, Calderon,
and Mayer \cite{adams:tilings}.

In addition, for a large class of biperiodic links whose quotient
admits a certain kind of alternating diagram on the torus, we find a
hyperbolic structure on the complement with a possibly incomplete
\emph{geometric triangulation}. That is, the complement decomposes
into tetrahedra that are all positively oriented with positive
volume. It is still unknown whether or not every link in $S^3$ admits
a geometric triangulation.  Moreover, we prove that this triangulation
is unimodular, and we find sharp lower and upper bounds for the
hyperbolic volume.  See Theorem~\ref{Thm:hyperbolic}.

The results on complete realizability have a number of interesting
consequences:

\subsubsection*{Exact volume and cusp shapes.}
We compute exact volumes for semi-regular links in terms of their
corresponding Euclidean tiling. Their cusp shapes can also be computed
from this tiling. See Theorem~\ref{Thm:BALvol} and
Corollary~\ref{Cor:CuspShapes}.

\subsubsection*{Commensurability and arithmeticity.}
Recall that two manifolds are \emph{commensurable} if they admit a
common finite sheeted cover. The \emph{trace field} of a hyperbolic
manifold $\HH^3/\Gamma$ is the smallest field containing the traces of
elements, and it is known to be a commensurability invariant of link
complements. However, families of links are known to be pairwise
incommensurable and yet have the same trace field,
e.g.\ \cite{ChesebroDeblois}. Using the geometry of semi-regular
links, we show that this phenomenon also holds for such links.
Infinitely many of them have trace field $\QQ(i,\sqrt{3})$ but are
pairwise incommensurable, assuming a conjecture of Milnor on the
Lobachevsky function.  Conversely, we also find infinitely many such
links that are commensurable to the figure-8 knot complement, and one
commensurable to the Whitehead link.  See Theorem~\ref{Thm:kM}.

\subsubsection*{Totally geodesic checkerboard surfaces.}
It is an open question whether any alternating knot admits two totally
geodesic checkerboard surfaces. Only a few links in $S^3$ are known to
admit totally geodesic checkerboard surfaces, including the Borromean
rings \cite{AdamsSchoenfeldII}. Pretzel knots $K(p,p,p)$ admit one
totally geodesic checkerboard surface, but not two
\cite{AdamsSchoenfeldI}.  Our geometric structures allow us to show
that two semi-regular links have totally geodesic checkerboard
surfaces, namely the infinite square weave and the triaxial link.
However, we prove that no other semi-regular alternating links have
this property. See Theorem~\ref{Thm:square_triaxial}.

\subsubsection*{Volume Density Conjecture.}
The volume density of a link in $S^3$ or $T^2\times I$ is defined to
be the ratio of its volume to crossing number. For a biperiodic
alternating link, the volume density is that of its quotient link in
$T^2\times I$.  In previous work, the authors showed that if a
sequence of knots or links in $S^3$ converges in an appropriate sense
(see Definition~\ref{def:folner_converge}) to the infinite square
weave, then their volume densities approaches the volume density of
the infinite square weave \cite{ckp:gmax}.  The \emph{Volume Density
  Conjecture} (Conjecture~\ref{conj:vol_density}) is that the same
result holds for any biperiodic alternating link.  In this paper, we
prove the Volume Density Conjecture for the triaxial link, which has
consequences for volume and determinant, as in \cite{ck:det_mp,
  ckp:density, ckp:gmax}.  See Theorem~\ref{Thm:triaxial}.

\

Elsewhere, biperiodic links have been called {\em tiling
  links}~\cite{adams:tilings}, {\em textile links}~\cite{BMOP} and
{\em textile structures}~\cite{morton-grishanov}.

\subsection{Organization}

In Section~\ref{sec:torihedra}, we describe the decomposition of the
link complement $(T^2\times I)-L$. In Section~\ref{sec:semiregular},
we prove that semi-regular links are completely realizable, and their
geometric structure and hyperbolic volume is determined by the
corresponding Euclidean tiling.  In Section~\ref{sec:arithmetic}, we
relate the geometry of the tiling to the commensurability,
arithmeticity and invariant trace fields of the corresponding
semi-regular links.  Sections~\ref{sec:triaxial}
and~\ref{sec:asymptotic} focus on two special links, namely the square
weave and the triaxial link. In Section~\ref{sec:triaxial}, we prove
that these are the only two semi-regular links that have totally
geodesic checkerboard surfaces.  In Section~\ref{sec:asymptotic}, we
prove the Volume Density Conjecture for the triaxial link.  Finally,
Section~\ref{sec:hyperbolic} is more broad. We prove that if any link
$L$ in $T^2\times I$ admits a certain kind of alternating diagram on
the torus, not just semi-regular, then $(T^2\times I)-L$ is hyperbolic
and admits a positively oriented, geometric triangulation.  We then
find sharp lower and upper bounds for the hyperbolic volume of these
links.

\subsection*{Acknowledgements}
We thank Colin Adams for useful discussions, and note that part (1) of
Theorem~\ref{Thm:BALvol} was proved independently in
\cite{adams:tilings}.  We thank the organizers of the workshops {\em
  Interactions between topological recursion, modularity, quantum
  invariants and low-dimensional topology} at MATRIX and {\em
  Low-dimensional topology and number theory} at MFO (Oberwolfach),
where part of this work was done.  We thank Tom Ruen, whose figures at
\cite{semi-regular-wiki} were very helpful for this project.  The
first two authors acknowledge support by the Simons Foundation and
PSC-CUNY. The third author acknowledges support by the Australian
Research Council.
%%%%%%%%%%%%%%%%%%%%%%%%%%%%%%%%%%%%%%%%%%%%%%%%%%%%%%%%%%%%%%%%%

\section{Torihedra}
\label{sec:torihedra}

Let $I=(-1,1)$. Let $L$ be a link in $T^2\times I$ with an alternating
diagram on $T^2\times\{0\}$, projected to the $4$--valent graph
$G(L)$.  First, we eliminate a few simple cases. If $G(L)$ is
contained in a disk in $T^2\times\{0\}$, then the link complement will
be reducible, with an essential sphere enclosing a neighborhood of the
disk. A similar argument shows that if a complementary region of
$G(L)$ is a punctured torus, the link complement is reducible. Hence,
we may assume that all complementary regions in
$(T^2\times\{0\})\setminus G(L)$ are disks or annuli.  We will say the
diagram is \emph{cellular} if the complementary regions are disks. We
call the regions the \emph{faces} of $L$ or of $G(L)$.  Similarly, if
$\L$ is the biperiodic link with quotient link $L$, then the faces of
$\L$ refer to the complementary regions of its diagram in $\R^2$,
which are the regions $\R^2-G(\L)$.  We will say the diagram of $L$ on
$T^2\times\{0\}$ is \emph{reduced} if four distinct faces meet at
every crossing of $G(\L)$ in $\R^2$.  Note that a reduced cellular
diagram has at least one crossing in $T^2\times\{0\}$, and an
alternating, reduced, cellular diagram has at least two crossings.
Throughout the paper, we will work with diagrams on $T^2\times \{0\}$
that are alternating, reduced and cellular.

\begin{definition}\label{Def:torihedron}
A \emph{torihedron} is a cone on the torus, i.e. $T^2\times
[0,1]/(T^2\times\{1\})$, with a cellular graph $G$ on
$T^2\times\{0\}$.  An \emph{ideal torihedron} is a torihedron with the
vertices of $G$ and the vertex $T^2\times\{1\}$ removed.  Hence, an
ideal torihedron is homeomorphic to $T^2\times [0,1)$ with a finite
  set of points (ideal vertices) removed from $T^2\times\{0\}$.
\end{definition}

Let $M=(T^2\times I)-L$ be the link complement.  The checkerboard
polyhedral decomposition of the complement of an alternating link in
$S^3$, as in \cite{menasco:polyhedra, ALR:polyhedra}, can be
generalized to that of an alternating link in $T^2\times I$. This was
done for a single example in \cite{ckp:gmax}, and much more generally
in \cite{HowiePurcell}.  The following version is appropriate for the
links considered here:

\begin{theorem}\label{Thm:Torihedra}
Let $L$ be a link in $T^2\times (-1,1)$ with a reduced, cellular,
alternating diagram on $T^2 \times \{0\}$.  Then the link complement
$M = (T^2\times I)-L$ admits a decomposition into two ideal torihedra
$M_1$ and $M_2$ with the following properties:
\begin{enumerate}
\item $M_1$ is homeomorphic to $T^2\times [0,1)$ and $M_2$ is
  homeomorphic to $T^2\times(-1,0]$, each with finitely many vertices
  removed from $T^2\times\{0\}$. The points $T^2\times\{\pm 1\}$ are
  ideal vertices denoted by $\pm\infty$.
\item $T^2\times\{0\}$ on $M_1$ and $M_2$ is labeled with the diagram
  graph $G(L)$, with vertices removed. Thus, the graph has 4-valent
  ideal vertices.
\item $M$ is obtained by gluing $M_1$ to $M_2$ along their faces on
  $T^2\times \{0\}$. Faces of $G(L)$ are checkerboard colored, and
  glued via a homeomorphism that rotates the boundary of each face by
  one edge in a clockwise or anti-clockwise direction, depending on
  whether the face is white or shaded.
\item Each edge class contains exactly four edges, and these edges are
  the crossing arcs of the link complement. At each ideal vertex, a
  pair of opposite edges are identified, in both $M_1$ and $M_2$, with
  the opposite pair identified at the same vertex in $M_1$ and
  $M_2$. See Figure~\ref{Fig:4valent-polyhedron}.
\end{enumerate}
\end{theorem}
We will refer to $M_1$ and $M_2$ as the upper and lower torihedra of
$L$.  When there is no ambiguity, the torihedron will be denoted by
$P(L)$.

\begin{proof}[Proof of Theorem \ref{Thm:Torihedra}]
We proceed exactly as in \cite{menasco:polyhedra}, but instead of two
3-balls, we start with a decomposition of $T^2\times I$ into two
thickened tori $T^2 \times [0,1)$ and $T^2 \times (-1,0]$.  The
torihedra are formed by cutting along checkerboard surfaces; an arc of
intersection between such surfaces becomes an ideal edge of the
decomposition.  The form of the torihedra near a crossing is
illustrated in \reffig{4valent-polyhedron}.  Since $L$ has an
alternating, reduced and cellular diagram, the 1-skeleton of each
torihedron is the same as $G(L)$, and its vertices are ideal
(removed).  The remainder of the construction for the cell
decompostion and the face identification in \cite{menasco:polyhedra}
is completely local at every crossing, and hence goes through likewise
in the toroidal case.
\end{proof}

\begin{figure}%[h]
  \centering
  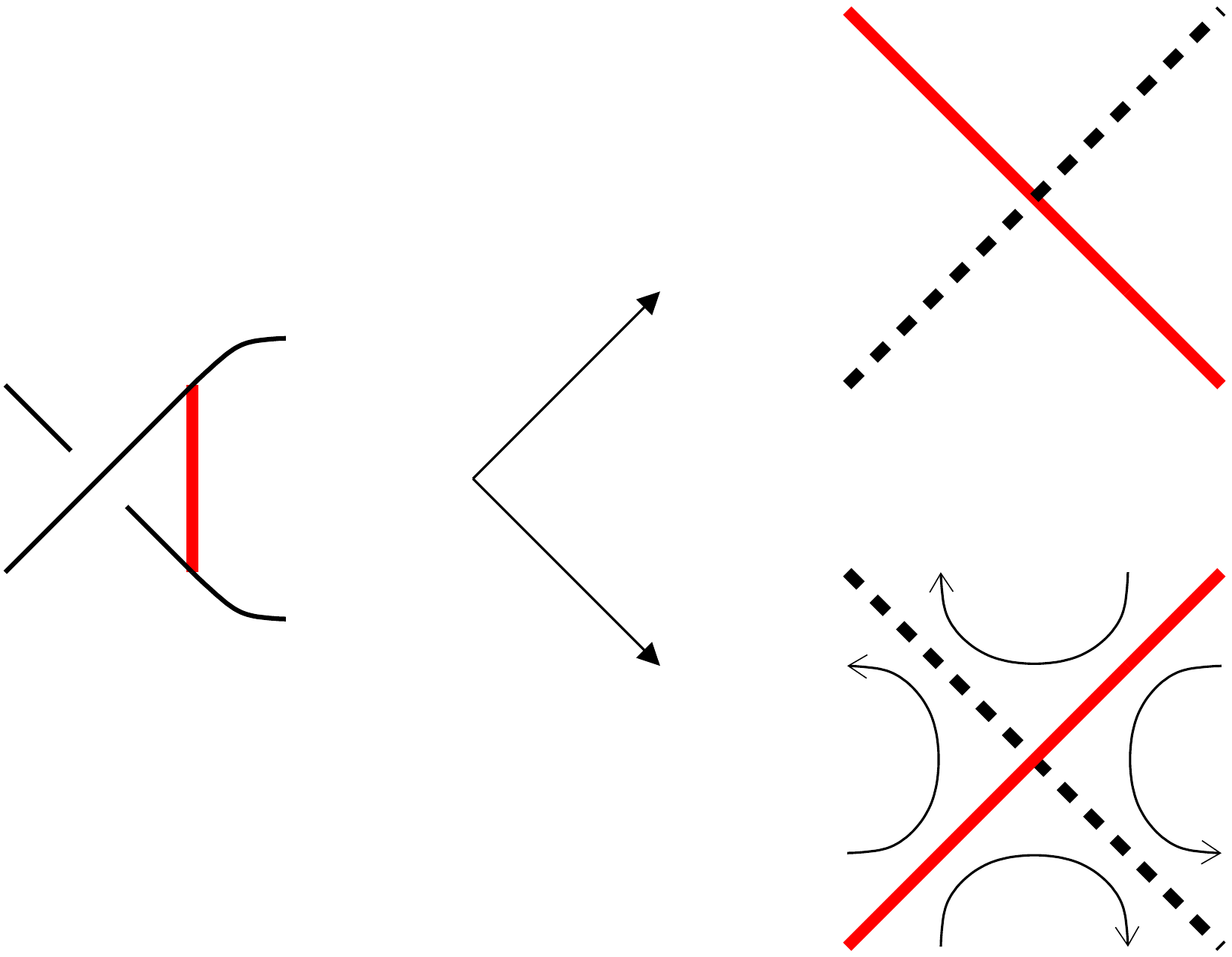
  \caption{Crossing edge $e$ at a $4$--valent vertex, and
    corresponding copies of $e$ on the upper and lower
    torihedra. The arrows on the lower diagram indicate the direction
    of twisting for gluing faces.}
  \label{Fig:4valent-polyhedron}
\end{figure}

These decompositions can be extended to biperiodic alternating links.

\begin{theorem}\label{Thm:BiperiodicDecomp}
Let $\L$ be a biperiodic alternating link whose quotient link $L$ in
$T^2\times I$ has a reduced, cellular, alternating diagram on
$T^2\times\{0\}$.  Then $\R^3-\L$ admits a decomposition into two
identical infinite ideal polyhedra.  Each infinite polyhedron $P(\L)$
is homeomorphic to $\R^2\times [0,\infty)$, with a biperiodic planar
graph identical to $G(\L)$ in $\R^2\times 0$ with vertices removed.
The polyhedra are glued via a homeomorphism determined by the
checkerboard coloring of faces of $G(\L)$ as in \refthm{Torihedra}.
\end{theorem}

\begin{proof}
Lift the torihedral decomposition of \refthm{Torihedra} to the cover
$\RR^3-\L$ of $(T^2\times I)-L$.
\end{proof}

\subsection{Geometric torihedra and stellation}

We now wish to determine conditions that guarantee the torihedra admit
a hyperbolic structure. In \cite{HowiePurcell} it was shown that if
such a link on $T^2\times I$ is weakly prime (see
Definition~\ref{Def:WeaklyPrime}) then the link complement is
hyperbolic. However, the torihedra themselsves may not admit a
hyperbolic structure. For example, if $G(\L)$ or, equivalently, $G(L)$
has bigon faces, there is no way the torihedra and polyhedra as
constructed can admit a hyperbolic structure. While geometric
information can be obtained even if the torihedra are not hyperbolic,
we will obtain much stronger results when they are. Thus we will
modify the torihedra by collapsing bigons.

\begin{definition}\label{Def:tiling}
Let $\L$ be a biperiodic alternating link whose quotient link $L$ in
$T^2\times I$ has a reduced, cellular, alternating diagram on
$T^2\times\{0\}$.  We define the graph $T_L$ on $T^2\times\{0\}$ by
collapsing all bigons of the diagram graph $G(L)$, as shown in
\reffig{bigons}.  If $L$ has no bigons, then $T_L=G(L)$.  By lifting
to the cover, we obtain a graph denoted $T_{\L}$ on
$\RR^2\times\{0\}\subset\RR^3$ by collapsing all bigons of $G(\L)$.
\end{definition}

\begin{lemma}\label{Lem:BigonCollapse}
Let $L$ be a link in $T^2\times I$ with a reduced, cellular,
alternating diagram on $T^2\times\{0\}$. Let $M_1$ and $M_2$ be the
torihedra in \refthm{Torihedra}. Collapse each bigon in $M_1$ and
$M_2$. Then $M=(T^2\times I)-L$ has a decomposition into ideal
torihedra with the following properties:
\begin{enumerate}
\item $T^2\times\{0\}$ on $M_1$ and $M_2$ contains a graph $T_L$ as in
  \refdef{tiling} with its vertices removed.  Thus, the valences of
  ideal vertices are $4$, $3$, and $2$, depending on whether the
  corresponding vertex of $G(L)$ is adjacent to $0$, $1$, or $2$
  bigons.
\item $M$ is obtained by gluing faces on $M_2$ to faces on $M_1$ by a
  homeomorphism that rotates the boundary of each face by a single
  rotation in one edge, clockwise or anti-clockwise depending on
  whether the same face in $G(L)$ is white or shaded.
\item Each edge class corresponds to a crossing arc as before, but
  crossing arcs associated to crossings adjacent across a bigon are
  identified into a single edge. Thus if a crossing is adjacent to no
  bigons, the corresponding edge has degree four. If it is adjacent to
  a twist region with $n$ bigons, the corresponding edge has degree
  $4+2n$.
\end{enumerate}

By lifting to the cover, a similar decomposition holds for $\RR^3-\L$,
with the graph $T_{\L}$ on $\RR^2\times\{0\}\subset\RR^3$.
\end{lemma}

\begin{proof}
The proof is by tracing the result of collapsing a bigon in torihedra
$M_1$ and $M_2$, similar to the process for collapsing bigons in the
polyhedral decomposition of an alternating link in $S^3$. This is
illustrated in \reffig{bigons}.
\end{proof}
  
\begin{figure}
  \centering
  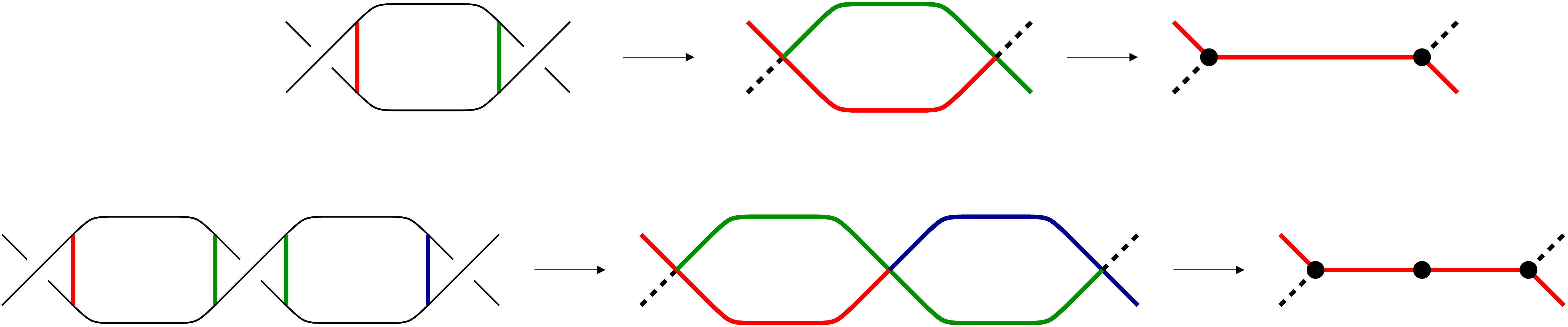
\caption{ Top: Collapsing a bigon creates an edge with 3--valent
  endpoints.  Bottom: Collapsing a sequence of bigons creates a chain
  of edges with 3--valent end vertices.}
  \label{Fig:bigons}
\end{figure}

We now turn the torihedral decomposition into a triangulation.

\begin{lemma}\label{Lem:Stellation}
Let $L$ be a link in $T^2\times I$ with a reduced, cellular,
alternating diagram on $T^2\times\{0\}$.  Then the link complement
$M=(T^2\times I)-L$ admits a decomposition into ideal tetrahedra with
the following properties:
\begin{enumerate}
\item Edges are labeled as \emph{horizontal}, \emph{vertical}, and
  \emph{stellating}. Each horizontal edge corresponds to an edge of
  $T_L$, and is identified to other horizontal edges. Each vertical
  edge runs from an ideal vertex of $M_1$ or $M_2$ to
  $\pm\infty$. Each stellating edge runs through the center of a
  non-triangular face in the complement of $T_L$, from $-\infty$ to
  $+\infty$.
\item For $n\geq 3$, each $n$--gon region in the complement of $T_L$
  on $T^2\times\{0\}$ is divided into $n$ ideal tetrahedra.
\end{enumerate}
\end{lemma}

Denote the collection of ideal tetrahedra of \reflem{Stellation} by
$\T$. We will call $\T$ the \emph{stellated bipyramid triangulation}.

\begin{proof}
Cone each face of $T_L$ on $T^2\times \{0\}$ to $\pm\infty$,
respectively, obtaining a bipyramid. Edges of the pyramids that are
incident to $\pm\infty$ become \emph{vertical edges}.  Edges of the
pyramids that are edges of $T_L$ become \emph{horizontal edges}.  See
Figure~\ref{Fig:hor-ver-edges}.

\begin{figure}%[h]
  \centering
  \input{figures/hex-face.pdf_tex}
  \caption{Horizontal, vertical and stellating edges of the stellated
    bipyramid triangulation.}
  \label{Fig:hor-ver-edges}
\end{figure}
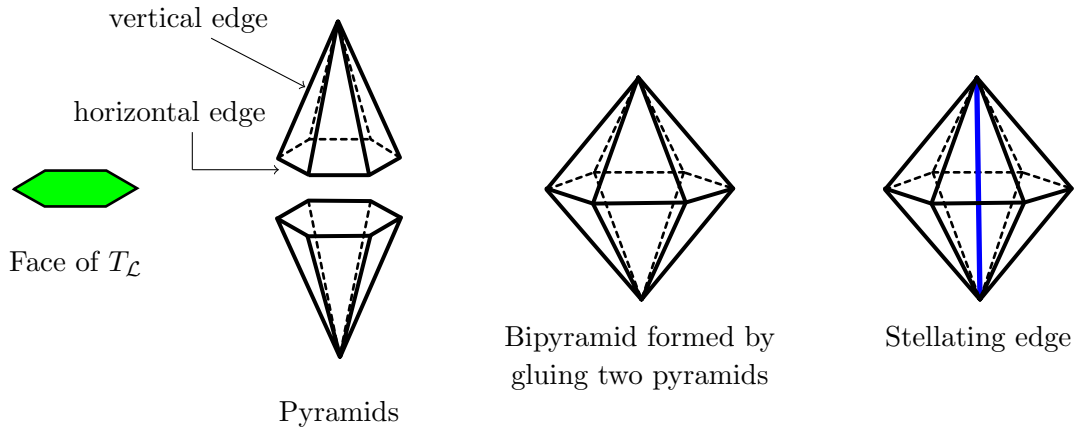

Now, for each face of $T_L$, glue the upper pyramid on that face to
the lower pyramid according to the checkerboard coloring, as described
in \reflem{BigonCollapse}.  This provides a decomposition of
$(T^2\times I)-L$ into bipyramids on the faces of $T_L$. Obtain an
ideal triangulation $\T$ of $(T^2\times I)-L$ by stellating the
bipyramids into tetrahedra by adding \emph{stellating edges}, as shown
in Figure~\ref{Fig:hor-ver-edges}.
\end{proof}

By taking the cover of $(T^2\times I)-L$, \reflem{Stellation} gives a
decomposition of $\RR^3-\L$ into ideal tetrahedra. This decomposition
was used in \cite{ckp:gmax} to triangulate the infinite square weave.
It was extended to alternating links in $S^3$ in
\cite{Adams:bipyramids} and to links in thickened higher genus
surfaces in \cite{adams:tilings}.  Our work here for links in
$T^2\times I$ was done independently.

%%%%%%%%%%%%%%%%%%%%%%%%%%%%%%%%%%%%%%%%%%%%%%%%%%%%%%%%%%%%%%%%%
\section{Semi-regular alternating links}
\label{sec:semiregular}

In this section, we study a class of biperiodic alternating links whose
geometric structure and hyperbolic volume can be immediately
determined from a corresponding Euclidean tiling.  For the tilings
considered in this section, the vertices can only be $3$--valent or
$4$--valent; i.e.\ the links $L$ or $\L$ can have at most one bigon
per twist region.

\begin{definition} 
A biperiodic alternating link $\L$ is called \emph{semi-regular} if
$T_{\L}$ is isomorphic, as a plane graph, to a biperiodic edge-to-edge
Euclidean tiling with convex regular polygons.  We also say the
quotient link $L$ is \emph{semi-regular}, and refer to the quotient
tiling of $T^2$ as $T_L$.
\end{definition}

Note that a semi-regular link is reduced and cellular. 

If the link has no bigons, then $T_{\L}$ is isomorphic to $G(\L)$.
Since there is at most one bigon in every twist region of $\L$, for
every bigon, both endpoints of the corresponding edge in $T_{\L}$ are
$3$--valent. Thus, if $V_4(T_{\L})$ denotes the subset of $4$--valent
vertices, then the graph $T_{\L}-V_4(T_{\L})$ admits a perfect
matching, defined as follows.

\begin{definition}\label{Def:PerfectMatching}
  A \emph{perfect matching} on a graph is a pairing of adjacent
  vertices that includes every vertex exactly once.
\end{definition}

Figure~\ref{Fig:3-isolated} shows an example of a tiling with isolated
$3$--valent vertices. Hence, this tiling cannot be associated with a
semi-regular biperiodic alternating link.

\begin{figure}
  \centering
  \includegraphics[height=1.6in]{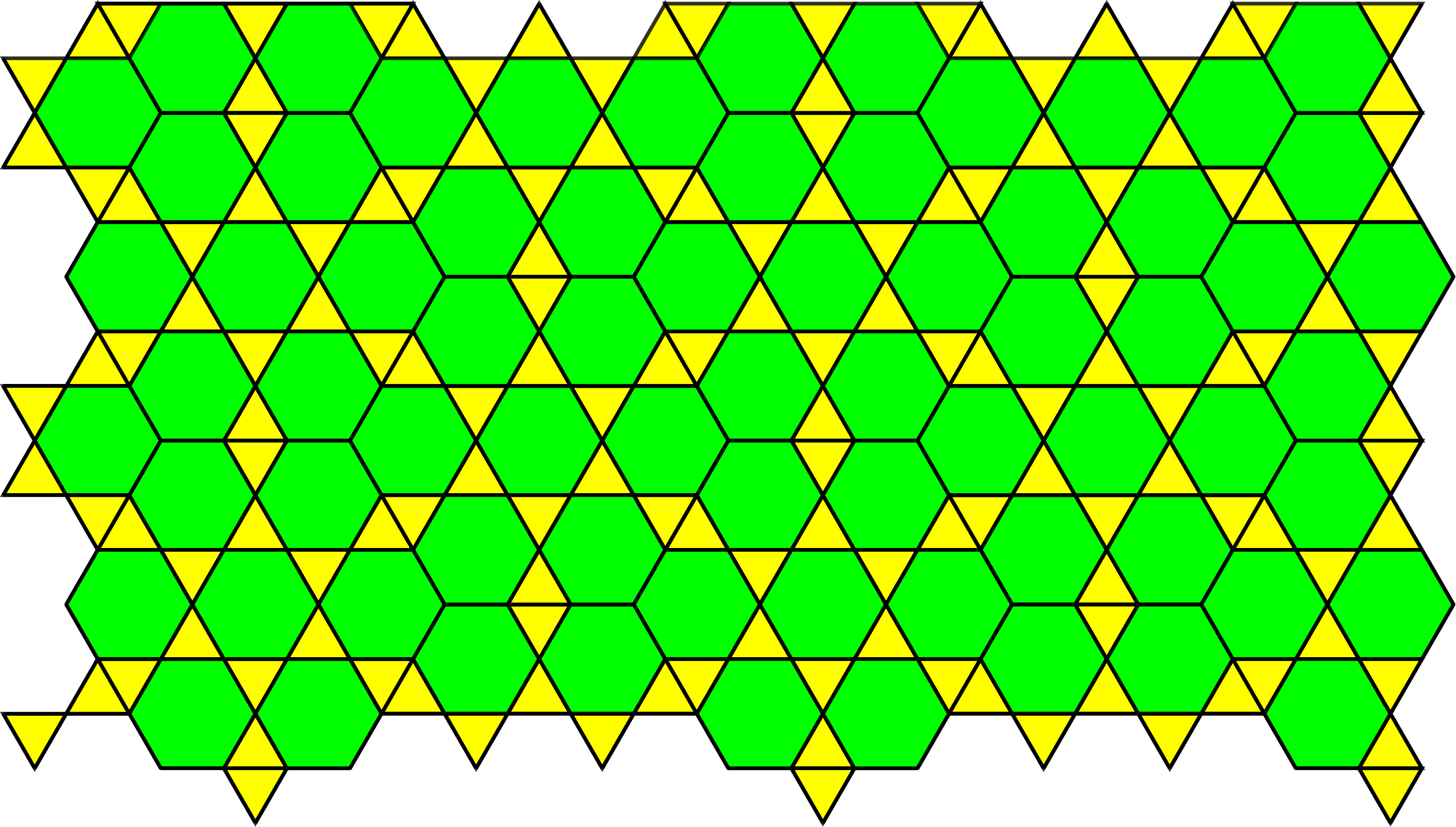}    \\
\caption{A tiling with isolated $3$--valent vertices. Figure modified from \cite{semi-regular-wiki}.}
\label{Fig:3-isolated}
\end{figure}

We now discuss local shapes for Euclidean tilings with $3$ and
$4$--valent vertices.  Let $T$ be a biperiodic edge-to-edge Euclidean
tiling with convex regular polygons. Gr\"{u}nbaum and Shephard
\cite{GS} classified such tilings according to the pattern of polygons
at each vertex. They described 21 vertex types which can occur in $T$.
If an $a$--gon, $b$--gon, $c$--gon, etc.\ appear in cyclic order
around the vertex, then its type is denoted by $a.b.c \ldots$.  For
example, the square tiling has vertex type $4.4.4.4$.

\begin{lemma}\label{lemma:4valent}
Let $T$ be any biperiodic edge-to-edge Euclidean tiling with convex
regular polygons, with vertices of valence $3$ and $4$, such that
$T-V_4(T)$ admits a perfect matching.
\begin{enumerate}
\item[($i$)]
  If $T$ has only $4$--valent vertices, then the only
  polygons that can occur in $T$ are triangles, squares and hexagons,
  and every vertex of $T$ has one of five vertex types:
  \[ 3.3.6.6,\quad   3.6.3.6,\quad   3.4.4.6,\quad  3.4.6.4, \ \text{or} \  4.4.4.4 \]
\item[($ii$)]
  If $T$ has only $4$--valent vertices, the number of
  triangles is twice the number of hexagons in the fundamental domain.
\item[($iii$)]
  If $T$ has $3$--valent vertices, then the only polygons
  that can occur in $T$ are triangles, squares, hexagons, octagons,
  and dodecagons.
\end{enumerate}
\end{lemma}

The allowable vertex types for part ($i$) are shown in \reffig{VertexTypes}.

\begin{figure}
\begin{tabular}{ccccc}
\includegraphics{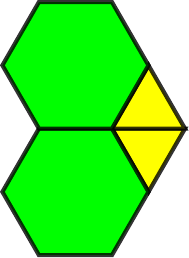}&
\includegraphics{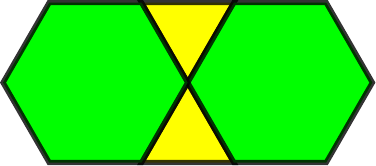}&
\includegraphics{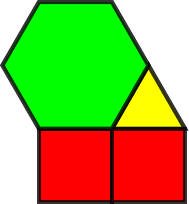}&
\includegraphics{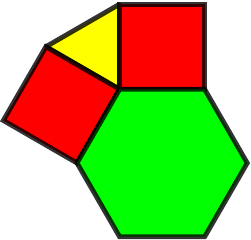} &
\includegraphics{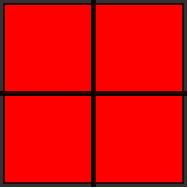} \\
$3.3.6.6$ &  $3.6.3.6$ &  $3.4.4.6$ & $3.4.6.4$ &  $4.4.4.4$ \\
\end{tabular}
\caption{Allowable vertex types for tilings with only $4$--valent
  vertices.}
\label{Fig:VertexTypes}
\end{figure}

\begin{proof}
Part~($i$):
Among the 21 vertex types classified in \cite{GS}, there
are only seven $4$--valent vertex types. These are the five types
discussed above, plus two more types that include a dodecagon:
$3.4.3.12$ and $3.3.4.12$.  But allowing only these seven types of
$4$--valent vertices makes it impossible to extend a tiling from a
vertex of the dodecagon to a tiling of the whole plane. A
straightforward case-by-case analysis shows that if a dodecagon is
present, then all the vertices of the dodecagon cannot be of type
$3.4.3.12$ or $3.3.4.12$, thus ruling out a dodecagon among the vertex
types for a $4$--valent tiling.

Part~($ii$): If $t,s,h$ are the numbers of triangles, squares and
hexagons, respectively, in the fundamental domain, then we have $4v =
2e = 6h+4s+3t$ and $f = h+s+t$.  Now, since the fundamental domain
gives a tiling of the torus, $v-e+f=0$ which implies that $t=2h$.

Part~($iii$): Again we consider the 21 vertex types classified in
\cite{GS}. Ten of these give a 3-valent vertex, and those of type
$4.8.8$, $6.6.6$, $3.12.12$, and $4.6.12$ satisfy the conditions of
the lemma. We show that the others cannot appear.  For each of the
other types, all but $5.5.10$ contain an $n$--gon that appears in no
other vertex types, namely $n=15$, $18$, $20$, $24$, and $42$. In each
of these cases, since we allow only 3 and 4-valent vertices, each of
the vertices meeting the $n$--gon must be of the same vertex type,
namely the unique 3-valent vertex meeting that $n$--gon. But now we
step through each case and check that no tiling of the plane exists
with these vertices on the $n$--gon, and only 3 and 4-valent vertex
types away from the $n$--gon. Finally, once we have ruled out the
other types, the same argument applies for the case $5.5.10$: all
vertices must be of type $5.5.10$, and we cannot complete the tiling.
\end{proof}

\begin{theorem}\label{Thm:Realization}
Every biperiodic edge-to-edge Euclidean tiling $T$ with convex regular
polygons, with vertices of valence $3$ or $4$, such that $T-V_4(T)$
admits a perfect matching, is the tiling $T_{\L}$ for a semi-regular
biperiodic link $\L$.
\end{theorem}

\begin{proof}
Given a tiling, note that if all vertices of $T$ are $4$--valent, then
$T=G(\L)=T_{\L}$ for a semi-regular link $\L$ with no bigons simply by
replacing vertices with crossings in an alternating manner.
Otherwise, if the tiling $T$ has $3$--valent vertices, then the
perfect matching condition implies that the $3$--valent vertices of
$T$ can be covered by a subset of mutually disjoint edges.  We double
each of these edges to obtain the projection $G(\L)$ of a semi-regular
link $\L$ with bigons.  Then the graph $T_{\L}$ of
\reflem{BigonCollapse} is exactly $T$.
\end{proof}

Let $\vtet\approx 1.01494$ and $\voct\approx 3.66386$ be the
hyperbolic volumes of the regular ideal tetrahedron and the regular
ideal octahedron, respectively.  Also let $v_{16}\approx 7.8549$ and
$v_{24}\approx 10.3725$ denote the volume of a hyperbolic ideal
bipyramid over a regular octagon and regular dodecagon, respectively,
with $16$ and $24$ faces.

The main result in this section is the following theorem. Part (1) was
proved independently in \cite{adams:tilings}.

\begin{theorem}\label{Thm:BALvol}
Let $\L$ be any semi-regular biperiodic link, with alternating
quotient link $L$ in $T^2\times I$, and $T_L$ as in \refdef{tiling}.
Then
\begin{enumerate}
\item\label{Itm:HypStruct} $(T^2\times I)-L$ has a complete hyperbolic
  structure coming from a decomposition into regular ideal bipyramids
  on the faces of $T_L$.
\item\label{Itm:4ValVol} If $T_L$ is $4$--valent, then the volume of
  $(T^2\times I)-L$, denoted $\vol(L)$, satisfies
  \[ \vol(L)= 10 H\,\vtet+S\,\voct, \]
  where the fundamental domain of $\L$ contains $H$ hexagons and $S$
  squares.
\item\label{Itm:3ValVol} If $T_L$ also has 3-valent vertices, then
  \[ \vol(L) = (6 H + 2T)\,\vtet + S\,\voct +  \Omega\,v_{16} + D\,v_{24}, \]
  where $H$, $T$, $S$, $\Omega$ and $D$ denote the number of hexagons,
  triangles, squares, octagons and dodecagons, respectively, in the
  fundamental domain of $\L$.
\end{enumerate}
\end{theorem}

The proof of the theorem uses angle structures. 

\begin{definition}\label{def:AngleStruct}
Given an ideal triangulation $\{\Delta_i\}$ of a 3--manifold, an
\emph{angle structure} is a choice of three angles $(x_i, y_i, z_i)
\in (0, \pi)^3$ for each tetrahedron $\Delta_i$, assigned to three
edges of $\Delta_i$ meeting in a vertex, such that
\begin{enumerate}
\item\label{item:sumtopi} $x_i + y_i + z_i = \pi$;
\item\label{item:oppedge} opposite edges in each $\Delta_i$ are
  assigned the same angle;
\item for every edge in the triangulation, the angle sum about the
  edge is $2\pi$.
\end{enumerate}
\end{definition}

Associated to any angle structure is a volume, obtained by summing the
volumes of hyperbolic ideal tetrahedra realizing the specified
angles. This defines a volume functional over the space of angle
structures for a given ideal triangulation of a manifold. The volume
functional is convex, and thus has a maximum. If the maximum of the
volume functional occurs in the interior of the space of angle
structures, then it follows from work of Casson and Rivin that the
angle structures at the maximum give the unique complete, finite
volume hyperbolic structure on the manifold \cite{rivin} (see also
\cite{futer-gueritaud:angles}). We will prove \refthm{BALvol} by
finding the unique maximum in the interior of the space of angle
structures.

We now slightly modify the stellated bipyramid triangulation $\T$ from
\reflem{Stellation} by performing a 3-2 move on every stellated
bipyramid over a trianglular face.  As a result, every triangle of
$T_{\L}$ corresponds to two ideal tetrahedra, one above and one below
the triangular face.  Let $\T'$ denote this modified triangulation.

\begin{lemma}\label{Lem:NoBigonsAngles}
Let $L$ be a semi-regular link with no bigons.  Then the triangulation
$\T'$ of $(T^2\times I)-L$ admits an angle structure with the
following properties:
\begin{enumerate}
\item The tetrahedra coming from stellated hexagons are all regular
  ideal tetrahedra.
\item The tetrahedra coming from each stellated octahedron glue to
  form a regular ideal octahedron.
\item The tetrahedra above and below every triangular face of $T_L$
  are regular ideal tetrahedra.
\end{enumerate}
\end{lemma}

\begin{proof}
Since $L$ has no bigons, $G(L)$ is isomorphic to $T_L$.  Recall that
all tetrahedra of $\T$ come from stellating bipyramids over all faces
of $T_L$, including the triangular faces.  We first assign angles to
vertical edges.  Each vertical edge of the triangulation runs from a
vertex of $G(L)$ to $\pm\infty$. The vertex is also a vertex of $T_L$,
a Euclidean tiling. The four faces adjacent to this vertex have four
interior angles coming from the Euclidean tiling. There are eight
tetrahedra coming from the stellation adjacent to this vertical edge:
two tetrahedra per face adjacent to that vertex in $T_L$. For each
tetrahedron, assign to its vertical edge half of the Euclidean
interior angle from $T_L$ of the associated face. Because $T_L$ is a
Euclidean tiling, it immediately follows that the angle sum around
every vertical edge is $2\pi$.

In addition, we obtain further information about adjacent faces.  A
regular Euclidean $n$--gon has interior angle $(n-2)\pi/n$.  If the
faces adjacent to a vertex of $T_L$ are an $n_1$--gon, an $n_2$--gon,
an $n_3$--gon, and an $n_4$--gon, with interior angles $\theta_1$,
$\theta_2$, $\theta_3$, and $\theta_4$, respectively, then the sum of
the interior angles at that vertex satisfies
\begin{align} \label{eq:4valent-anglesum}
\sum_{i=1}^4 \theta_i = \sum_{i=1}^4 \frac{(n_i-2)\pi}{n_i} = 2\pi
\quad \implies \quad \frac{1}{n_1} + \frac{1}{n_2} +\frac{1}{n_3}
+\frac{1}{n_4} &=1.
\end{align}
See \reffig{4valent-anglesum}~(a).

\begin{figure}
  \centering
  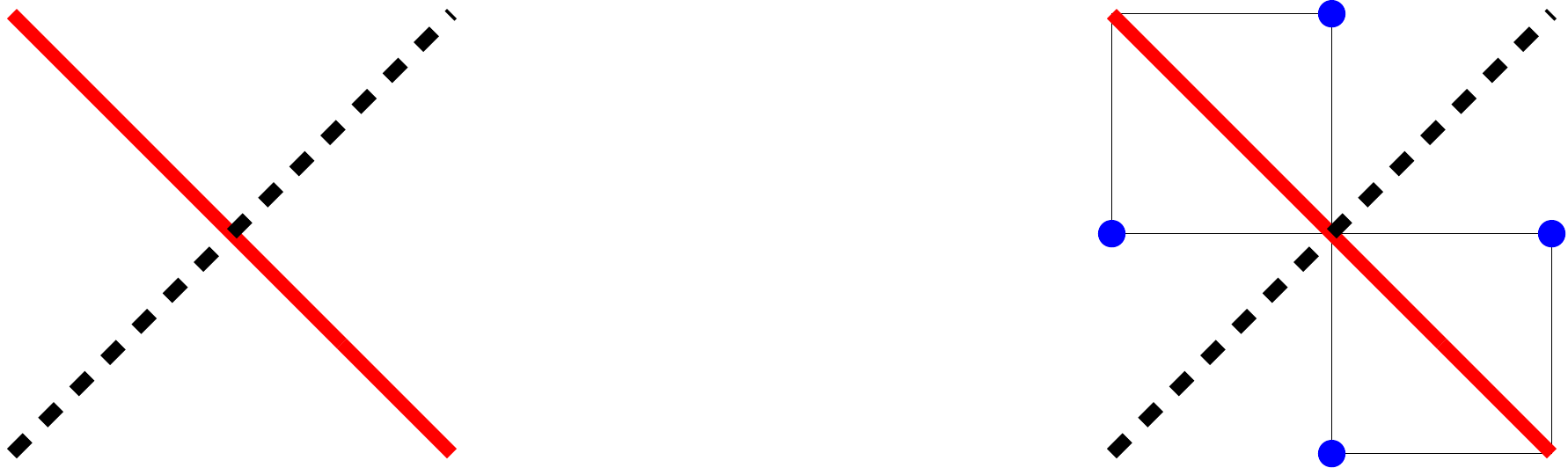
  \caption{(a) Vertex angles $\theta_i$ at a 4--valent vertex of the
    tiling $T_L$.  (b) Dihedral angles $\alpha_i=2\pi/n_i$ at
    stellating edges of $\T$.}
\label{Fig:4valent-anglesum}
\end{figure}

Next we assign angles to the stellating edges. These run from
$-\infty$ to $+\infty$ through the center of an $n$--gon face, and are
adjacent to $n$ tetrahedra.  As the faces of $T_L$ are regular
Euclidean polygons, we assign each edge meeting the stellating edge an
angle $2\pi/n$, so the angle sum around stellating edges is
$2\pi$. Because opposite dihedral angles on each tetrahedron are
equal, this also assigns angles to horizontal edges: each has angle
$2\pi/n$.

Now consider the angle sum over all angles assigned to edges
identified to a fixed horizontal edge $e$. There are exactly four
tetrahedra incident to $e$, shown in \reffig{4valent-anglesum}~(b),
with one tetrahedron in each of the four faces adjacent to a vertex of
$T_L$. The dihedral angle on the horizontal edge coming from the
$i$-th tetrahedron is $\alpha_i=2\pi/n_i$.  Thus, the angle sum at $e$
is
\[ \frac{2\pi}{n_1} + \frac{2\pi}{n_2} + \frac{2\pi}{n_3} + \frac{2\pi}{n_4} = 2 \pi,\]
where the equality follows by equation~(\ref{eq:4valent-anglesum}).
Therefore, the angles assigned to the horizontal, vertical and
stellating edges make the angle sum $2\pi$ on every edge of $\T$,
which gives an angle structure.

Finally, we use this angle structure to get an angle structure on
$\T'$.  Stellated bipyramids over a triangular face consist of three
tetrahedra, each with angle $2\pi/3$ on their horizontal and
stellating edges, and angle $\pi/6$ on vertical edges. Replacing these
three tetrahedra via a 3-2 move, the stellating edge is removed.  At
vertical edges, two faces glue to one, so the angle becomes
$\pi/3$. At horizontal edges, the angle is halved, becoming
$\pi/3$. This results in two ideal tetrahedra with all angles $\pi/3$,
which are regular ideal tetrahedra.
\end{proof}

\begin{lemma}\label{Lem:BigonsAngles}
Let $L$ be a semi-regular link with at most one bigon per twist
region.  Then the triangulation $\T'$ of $(T^2\times I)-L$ admits an
angle structure with the properties of \reflem{NoBigonsAngles}, and
additionally
\begin{enumerate}
\item[(4)] Tetrahedra coming from octagons glue to form a bipyramid over a
  regular ideal octagon, and those coming from dodecagons form a
  bipyramid over a regular ideal dodecagon.
\end{enumerate}
\end{lemma}

\begin{proof}
As in \reflem{NoBigonsAngles}, we stellate bipyramids, and assign
angles to vertical edges of tetrahedra using the vertex angles of the
Euclidean tiling $T_L$. To stellating edges, hence also to horizontal
edges, assign angles $2\pi/n$.  Again, we need to check that the angle
sum around each edge class is $2\pi$.  For vertical and stellating
edges, and for horizontal edges with no 3-valent endpoints, the proof
of \reflem{NoBigonsAngles} applies as before to give angle sum $2\pi$
as desired.

Consider a horizontal edge $e$ with 3-valent endpoints.  Recall the
effect of collapsing bigons, as in \reflem{BigonCollapse}: the
horizontal edge $e$ will be identified to the crossing arcs of a bigon
face, as shown in \reffig{bigons}~(top).  There are four faces
surrounding the two 3-valent vertices of $e$: an $n_1$--gon, an
$n_2$--gon, an $n_3$--gon, and an $n_4$--gon, as illustrated in
\reffig{3valent-anglesum}~(a).

\begin{figure}
  \centering
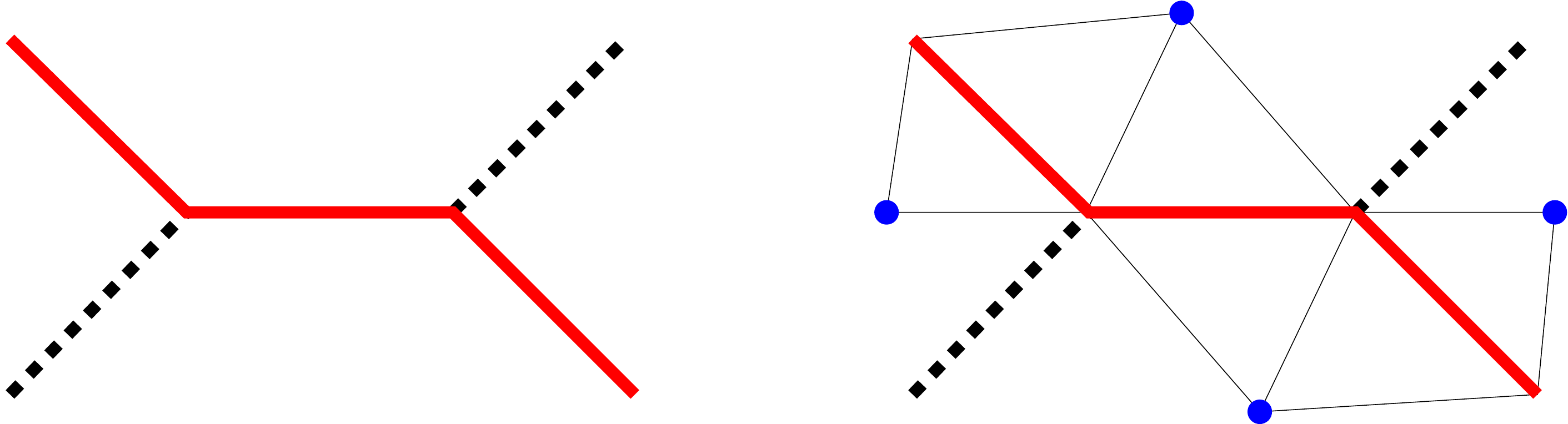
\caption{ (a) Vertex angles $\theta_i$ at a pair of 3--valent vertices
  of the tiling $T_L$.  (b) Dihedral angles $\alpha_i=2\pi/n_i$,
  $1\leq i \leq 4$, at stellating edges of $\T$.}
  \label{Fig:3valent-anglesum}
\end{figure}

The vertex angle of each regular Euclidean $n_i$--gon is
$\theta_i=\frac{(n_i-2)\pi}{n_i}$ for $1\leq i \leq 4$.  Hence, the
following equations express the angle sum around the two endpoints of
edge $e$:

\begin{align*}
  \frac{(n_1-2)\pi}{n_1} + \frac{(n_2-2)\pi}{n_2} +
  \frac{(n_3-2)\pi}{n_3} = 2 \pi & \quad\implies\quad
  \frac{2}{n_1} + \frac{2}{n_2} + \frac{2}{n_3} = 1
  \\
  \frac{(n_1-2)\pi}{n_1} + \frac{(n_3-2)\pi}{n_3} +
  \frac{(n_4-2)\pi}{n_4} = 2 \pi &\quad\implies\quad
  \frac{2}{n_1} + \frac{2}{n_3} + \frac{2}{n_4} = 1
\end{align*}
Adding the equations on the right, we obtain:
\begin{align}
 \label{eq:3valent-anglesum}
 \frac{2}{n_1} + \frac{1}{n_2} + \frac{2}{n_3} + \frac{1}{n_4} &= 1
\end{align}

Now, the horizontal edge $e$ associated to the crossing arc of a bigon
is identified to six tetrahedra, as shown in
\reffig{3valent-anglesum}~(b). The angle sum at edge $e$ is
\[
\frac{4\pi}{n_1} + \frac{2\pi}{n_2} + \frac{4\pi}{n_3} +
\frac{2\pi}{n_4} = 2 \pi,
\]
where the equality follows by Equation~(\ref{eq:3valent-anglesum}).
Thus, this angle assignment gives the desired angle structure.

To complete the proof, perform a 3-2 move on the stellated bipyramids
over triangles, as in the proof of \reflem{NoBigonsAngles}.  As
before, we obtain two regular ideal tetrahedra for each triangular
face of $T_L$.
\end{proof}

\begin{lemma}\label{Lem:VolMax}
Let $L$ be a semi-regular link.  For the triangulation $\T'$ of
$(T^2\times I)-L$, let $\A(\T')$ be the set of angle structures on
$\T'$. Then the volume functional $\vol\from\A(\T')\to\RR$ is
maximized at the angle structure of \reflem{NoBigonsAngles} or
\reflem{BigonsAngles}, depending on whether or not $T_L$ has 3-valent
vertices.
\end{lemma}

\begin{proof}
If $L$ has no bigons, then by
Lemma~\ref{lemma:4valent}~($i$), the only regular
$n$--gons that can arise in $T_L$ have $n=3,\,4$, or $6$.  If $L$ has
bigons, then $3$--valent vertices appear in pairs, and octagons and
dodecagons can also arise.

The tetrahedra coming from regular triangles and hexagons have the
angles of a regular ideal tetrahedron. Since the volume of a
tetrahedron is maximized by the regular ideal tetrahedron, the given
angle structure maximizes the volume of these tetrahedra.

For $n=4$, $8$, or $12$, we claim that the tetrahedra obtained by
stellating the regular $4$, $8$, and $12$--bipyramids also maximize
volume.  The regular $4$--bipyramid is the regular ideal octahedron,
which has volume $\voct$.  If we combine the tetrahedra obtained by
stellating the square bipyramids, we obtain a decompostion of
$(T^2\times I) -L$ into regular tetrahedra and octahedra.  In
\cite[Lemma~3.3]{ckp:weaving}, we proved that for any angle structure
on an ideal octahedron $P$, the volume of that angle structure
satisfies $\vol(P) \leq \voct$.  The same argument applies as well for
$n=8$ and $n=12$: by work of Rivin~\cite{rivin}, an angle structure
will have volume bounded by the unique complete hyperbolic structure
on the $n$--bipyramid with that angle assignment. But the maximal
volume of a complete ideal $n$--bipyramid is obtained uniquely by the
regular $n$--bipyramid (see
\cite[Theorem~2.1]{Adams:BipyramidVolume}).

Therefore, whether or not $L$ has bigons, the angle structure on $\T'$
maximizes volume.
\end{proof}

\begin{proof}[Proof of \refthm{BALvol}]
In Lemmas~\ref{Lem:NoBigonsAngles} and \ref{Lem:BigonsAngles}, we
found an angle structure on an ideal triangulation of the complement
of any semi-regular biperiodic link, and by \reflem{VolMax}, that
angle structure maximizes volume. Work of Casson and Rivin implies
that the gluing of hyperbolic ideal tetrahedra that realize this angle
structure gives the complete finite volume structure on
$(T^2\times I)-L$. This completes the proof of part
\eqref{Itm:HypStruct}.

To prove parts \eqref{Itm:4ValVol} and \eqref{Itm:3ValVol}, it remains
to apply Lemma~\ref{lemma:4valent}, describing which Euclidean tilings
can occur as $T_{\L}$.  For part \eqref{Itm:4ValVol}, by
Lemma~\ref{lemma:4valent}~($ii$), there are twice the number of
triangles as hexagons.  A bipyramid on a regular hexagon decomposes
into six regular ideal tetrahedra, which contribute $6\vtet$ to the
volume.  Since every triangle contributes two regular ideal
tetrahedra, it follows that every hexagon contributes $10\vtet$ to the
volume. Thus, if $T_{\L}$ has $H$ hexagons and $S$ squares per
fundamental domain, then the volume density of $\L$ is $10
H\,\vtet+S\,\voct$ per fundamental domain.

Similarly for part \eqref{Itm:3ValVol}, for each regular $n$--gon face
of $T_{\L}$, the contribution to the volume is that of a regular
$n$--bipyramid.  This proves the result.
\end{proof}

\begin{corollary}\label{Cor:CuspShapes}
Let $\L$ be any semi-regular biperiodic link, with alternating
quotient link $L$ in $T^2\times I$. Then the cusps of
$(T^2\times I)-L$ satisfy the following.
\begin{enumerate}
\item Cusps corresponding to $T^2\times \{\pm 1\}$ have fundamental
  domain identical to a fundamental domain of the corresponding
  Euclidean tiling $T_{\L}$.
\item Cusps corresponding to components of $L$ are tiled by regular
  triangles and squares in the case $L$ has no bigons, and otherwise
  by regular triangles and squares, and triangles obtained by
  stellating a regular octagon and a regular dodecagon.
\end{enumerate}
\end{corollary}

\begin{proof}
The second item follows from the fact that $(T^2\times I)-L$ is obtained by gluing regular tetrahedra and octahedra, and the tetrahedra obtained by stellating bipyramids over regular octagons and dodecagons in the case of bigons. The first follows from the fact that the pattern of such polygons meeting the cusps corresponding to $T^2\times\{\pm 1\}$ comes from the Euclidean tiling.
\end{proof}

%%%%%%%%%%%%%%%%%%%%%%%%%%%%%%%%%%%%%%%%%%%%%%%%%%%%%%%%%%%%%%%%%
\section{Commensurability and arithmeticity of semi-regular links}
\label{sec:arithmetic}

The proof of Theorem~\ref{Thm:BALvol} enables us to compute the
invariant trace fields of semi-regular links without bigons.  In
Theorem~\ref{Thm:kM} below, we relate the geometry of the tiling to
the commensurability, arithmeticity and invariant trace fields of the
corresponding links.  Milnor \cite{milnor} conjectured that, except
for certain well-known relations, values of the Lobachevsky function
at rational multiples of $\pi$ are rationally independent.  In
particular, its values at $\pi/3$ and $\pi/4$ are conjectured to be
rationally independent.  Assuming this, we show that there exist
infinitely many semi-regular links that are pairwise incommensurable,
but have the same invariant trace field.

\begin{theorem}\label{Thm:kM}
For a semi-regular link $\L$ with no bigons, with alternating quotient
link $L$, let $M=(T^2\times I)-L$ and let $k(M)$ denote its invariant
trace field.
\begin{enumerate}
\item\label{Itm:Squares} If the fundamental domain of $T_{\L}$
  contains only squares, then $k(M)= \Q(i)$, and $M$ is commensurable
  to the Whitehead link complement.  Hence, $M$ is arithmetic.  In
  this case, $\L$ is the unique semi-regular link called the square
  weave $\W$ below.
\item\label{Itm:TriHex} If the fundamental domain of $T_{\L}$ contains
  only triangles and hexagons, then $k(M)=\Q(i\sqrt{3})$, and $M$ is
  commensurable to the figure-8 knot complement.  Hence, $M$ is
  arithmetic.  In this case, $\L$ is one of infinitely many
  semi-regular links. Examples are shown in Figures~\ref{Fig:trihex}
  and \ref{Fig:arithmetic}~(left).
\item\label{Itm:Mix} If the fundamental domain of $T_{\L}$ contains at
  least one hexagon and one square, then $k(M)= \Q(i,\,\sqrt{3})$.
  Hence $M$ is not arithmetic.  Assuming $\vtet$ and $\voct$ are
  rationally independent, there are infinitely many commensurability
  classes of semi-regular links with this invariant trace field
  (e.g.\ see Figure~\ref{Fig:arithmetic}~(right)).
  \end{enumerate}
\end{theorem}

\begin{figure}%[h]
  \centering
\includegraphics[height=1 in]{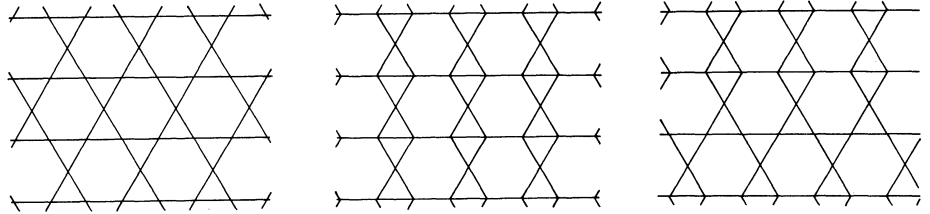}
\caption{Figure 4 from \cite{GS} showing three of the infinitely many
  semi-regular links with triangles and hexagons}
  \label{Fig:trihex}
\end{figure}

\begin{figure}
\begin{center}
  \includegraphics[height=1.5in]{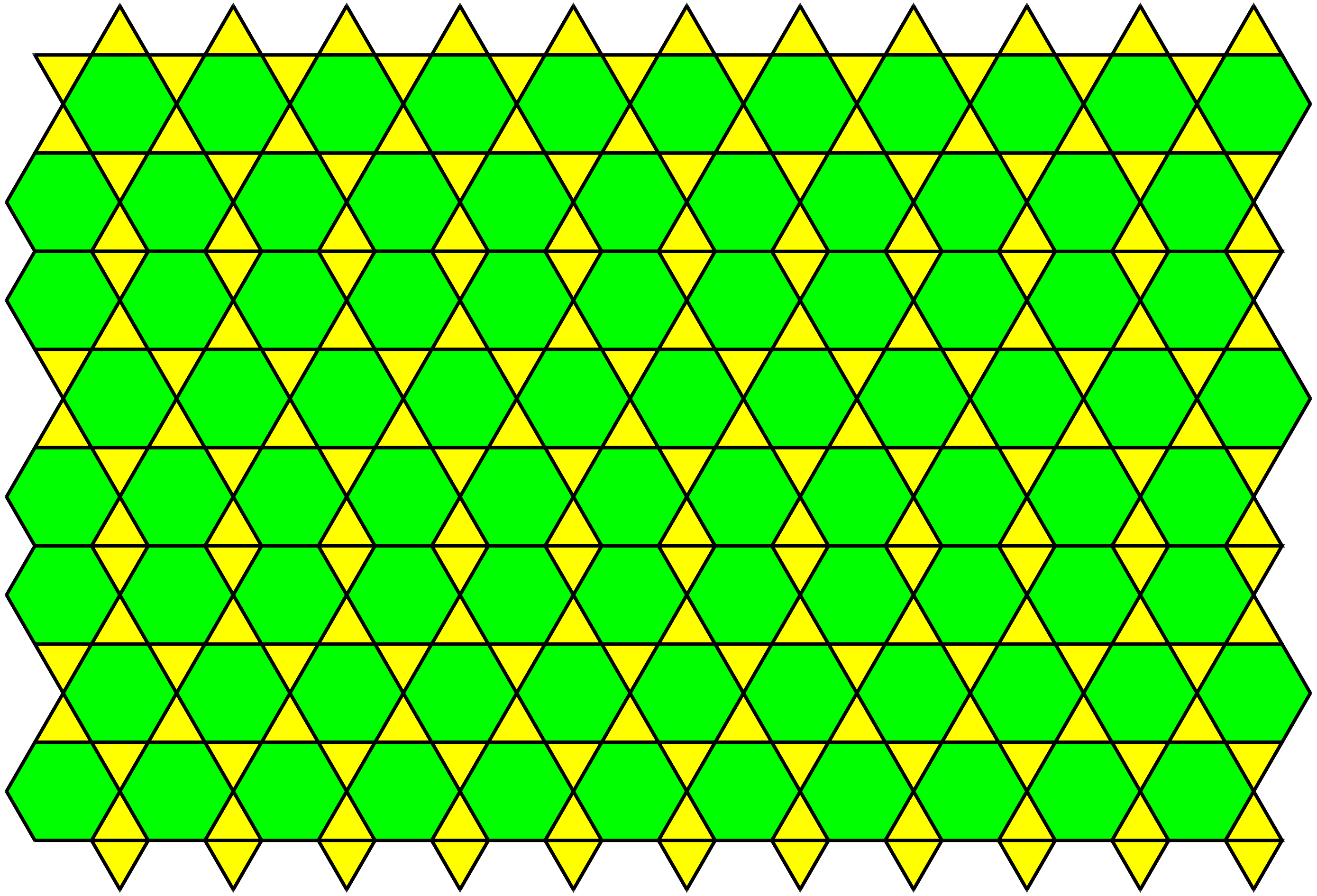} \qquad
  \includegraphics[height=1.5in]{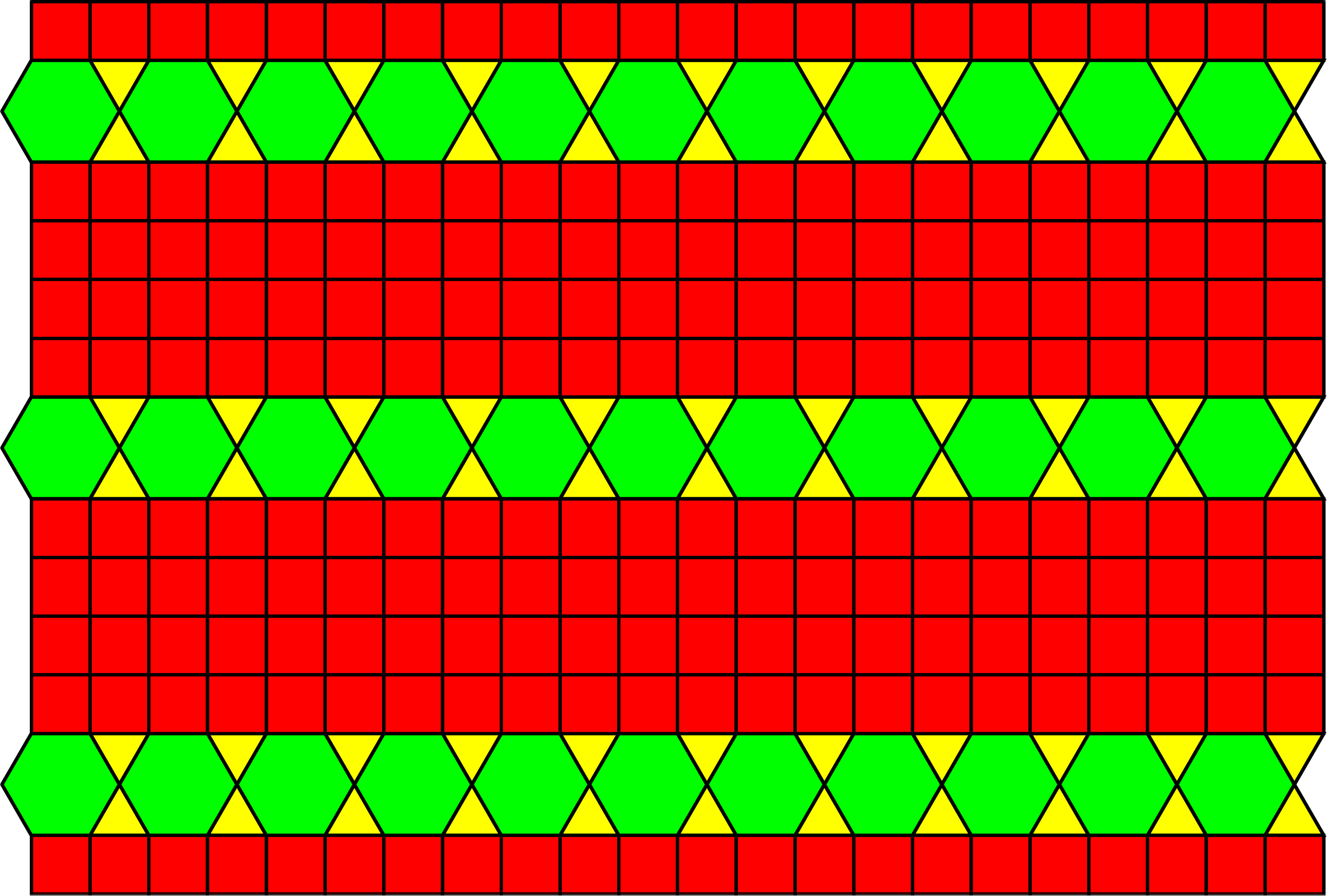}
  \caption{Left: Tiling for one of the infinitely many semi-regular
    links commensurable with the figure-8 knot.  Right: Tiling for
    $\L_2$, one of a pairwise incommensurable family of semi-regular
    links $\L_i$ with the same invariant trace field.  Figures
    modified from \cite{semi-regular-wiki}.}
\label{Fig:arithmetic}
\end{center}
\end{figure}

\begin{proof}
We first compute the invariant trace fields in all the three
cases. Another way of looking at the proof of
Theorem~\ref{Thm:BALvol}~(1) is that the edge gluing equations for
$\T'$ have solutions with tetrahedral parameters $e^{i\pi /3}$ and/or
$e^{i\pi/2}$.  Since the invariant trace field for cusped manifolds is
generated by the tetrahedral parameters \cite{NR92}, the invariant
trace fields are as given above.

We now address the commensurability of these links.  In case
\eqref{Itm:Squares}, a direct computation in Snap \cite{snap} verifies
that one such $M$ is arithmetic, namely the quotient of the square
weave. Any link with fundamental domain consisting only of squares
must be commensurable to this one.  Since $M$ is cusped, has finite
volume, is arithmetic, and has the same invariant trace field as the
Whitehead link complement, $M$ is commensurable to the Whitehead link
complement (see e.g.~\cite[Theorem~8.2.3]{MR2003}).

For case \eqref{Itm:TriHex}, since $T_{\L}$ contains only triangles
and hexagons, it uses tiles only of the type $3.6.3.6$ or $3.3.6.6$ by
Lemma~\ref{lemma:4valent}. As observed in Section 1 of \cite{GS}, by
stacking horizontal strips made up of only one type of tile, we can
obtain infinitely many biperiodic tilings; examples are shown in
Figures~\ref{Fig:trihex} and \ref{Fig:arithmetic}~(left).

Moreover, in any biperiodic tiling with just triangles and hexagons,
such a horizontal strip always exists, and translation along the strip
is one of the directions of the biperiodic action.  Each strip of the
tiles $3.6.3.6$ or $3.3.6.6$ consists of strips of four types of
parallelograms, each of which consists of half of the hexagon and a
triangle, as shown in \reffig{trihex-comm}. These four
parallelograms are related by reflections and $\pi$--rotations.
Since we are using only tiles of the type $3.6.3.6$ or $3.3.6.6$, we
can construct the biperiodic tiling using just these four
parallelograms. Because they are all related by reflections and
$\pi$--rotations, this implies all such tilings are commensurable.

\begin{figure}
  \centering
\begin{tikzpicture}
 \node at (3,0)
 { \includegraphics[width=1.85in]{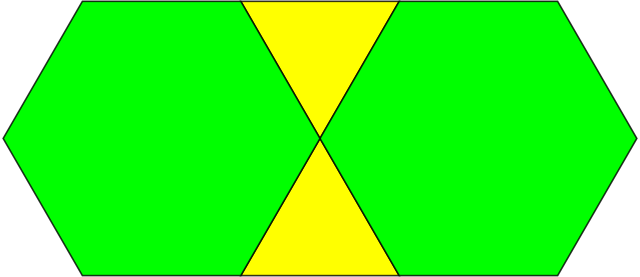}};
\node at (-4,0)
{ \includegraphics[height=1.2 in]{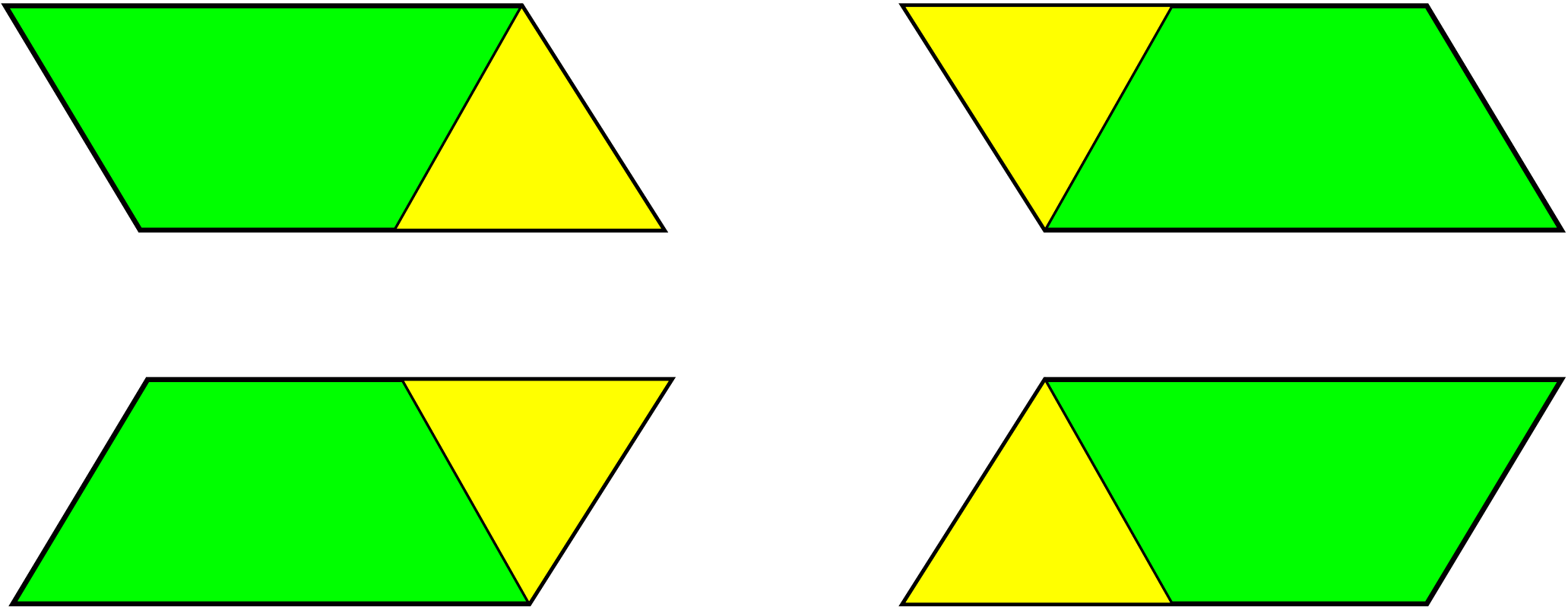}};
\draw[blue, thick] (.65,0)--(3,0)--(3.6,1)--(1.25,1)--(.65,0);
\end{tikzpicture}
\caption{Left: Four parallelograms which comprise strips of the tiles
  $3.6.3.6$ and $3.3.6.6$. Right: One of the parallelograms fitting on
  the tile $3.6.3.6$. }
  \label{Fig:trihex-comm}
\end{figure}

A direct computation in Snap \cite{snap} verifies that one of these
links is arithmetic: namely the complement of the quotient of the
triaxial link $M$ (see Figure~\ref{Fig:triaxial}). Since $M$ is
cusped, has finite volume, is arithmetic, and has the same invariant
trace field as the figure-8 knot complement, $M$ is commensurable to
the figure-8 knot complement \cite{MR2003}.
Hence the commensurability claim follows for all biperiodic tilings
containing only triangles and hexagons.

For case \eqref{Itm:Mix}, let $\L_1$ and $\L_2$ be any two
semi-regular links with no bigons, and at least one hexagon and one
square per fundamental domain.  By Theorem~\ref{Thm:BALvol}~(2),
$\vol(L_1)=p_1\,\vtet+q_1\,\voct$ and
$\vol(L_2)=p_2\,\vtet+q_2\,\voct$ for certain positive integers $p_1,
q_1, p_2, q_2$.

If $\L_1$ and $\L_2$ are commensurable, then there exist integers
$A$ and $B$ such that
\[
A(p_1\,\vtet+q_1\,\voct)=B(p_2\,\vtet+q_2\,\voct) \implies (A\,
p_1-B\, p_2)\vtet = (B\, q_2-A\, q_1)\voct.
\]
Assuming $\vtet$ and $\voct$ are rationally independent, this equation
implies that
\[
\frac{A}{B}=\frac{p_2}{p_1} \quad \text{and} \quad
\frac{A}{B}=\frac{q_2}{q_1} \implies p_1\, q_2 - q_1\, p_2 = 0.
\]
By \refthm{Realization}, there exist infinitely many semi-regular
links $\L_i$ with no bigons, such that for any pair
$\L_{i_1},\,\L_{i_2}$,
\[ p_{i_1}\, q_{i_2} - q_{i_1}\, p_{i_2} \neq 0.\]
Hence these links $\L_i$ are pairwise incommensurable.
\end{proof}

Figure~\ref{Fig:arithmetic}~(right) shows the tiling for $\L_2$, which
is part of a family $\L_j$ with one hexagon and $4j$ squares per
fundamental domain.  All the $\L_j$'s have the same invariant trace
field.

%%%%%%%%%%%%%%%%%%%%%%%%%%%%%%%%%%%%%%%%%%%%%%%%%%%%%%%%%%%%%%%%%
\section{The triaxial link}
\label{sec:triaxial}

The triaxial link $\L$ is shown in Figure~\ref{Fig:triaxial} with its
projection, the trihexagonal tiling.  It has long been used for
weaving, and appears to have been considered mathematically by Gauss,
who drew it in his 1794 notebook; see \cite{Przytycki}.

\begin{figure}
  \centering 
  \includegraphics{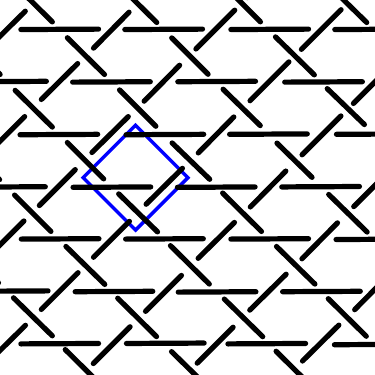}\quad
  \includegraphics{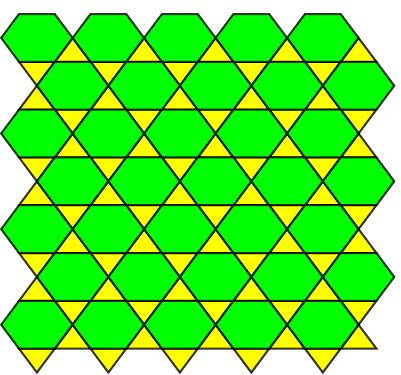}\quad
  \includegraphics[height=1.5in]{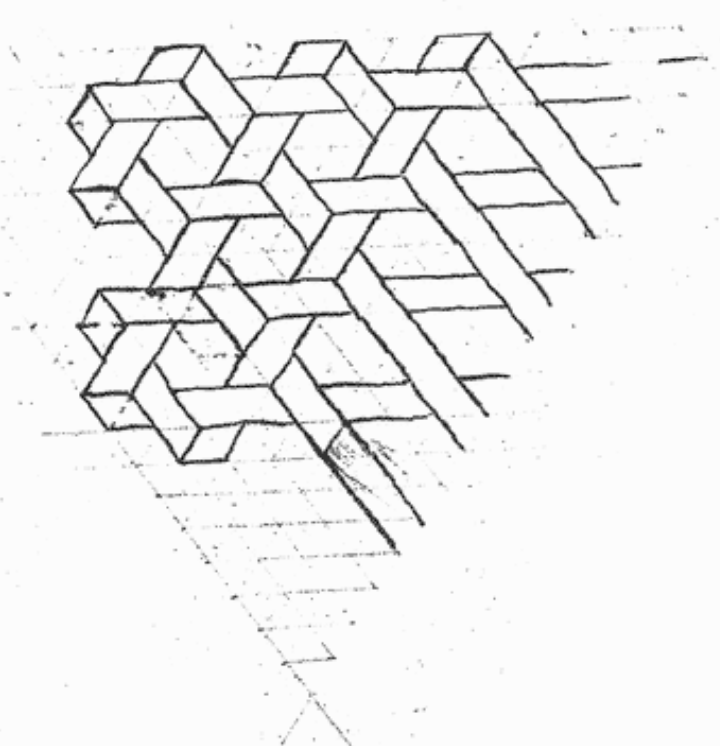}
  \caption{Left to right: Triaxial link $\L$ with a fundamental domain
    (blue square), trihexagonal tiling $T_{\L}$, and a figure from
    Gauss' 1794 notebook \cite{Przytycki}.}
  \label{Fig:triaxial}
\end{figure}

Let $L$ be the alternating quotient of the triaxial link in $T^2\times
I$.  This can be described as a link $K$ in $S^3$ by drawing a
fundamental domain on a Heegaard torus in $S^3$, then adding the Hopf
link given by the cores of the two Heegaard tori, as in
Figure~\ref{Fig:5-chain} (left). After isotopy, the diagram of $K$
appears as in Figure~\ref{Fig:5-chain} (center). This link complement
is isometric to the complement of the minimally twisted 5-chain link,
shown in Figure~\ref{Fig:5-chain} (right), although the links are not
isotopic: $K$ is the link $L12n2232$, and the minimally twisted
5-chain link is $L10n113$ in the Hoste-Thistlethwaite census of links
up to 14 crossings \cite{Hoste-enumeration}.  Among its many
interesting properties, $S^3-K$ is conjectured to be the $5$--cusped
manifold with the smallest hyperbolic volume, and most of the
hyperbolic manifolds in the cusped census can be obtained as its Dehn
fillings~\cite{MPR}.

\begin{figure}
  \centering
  \includegraphics{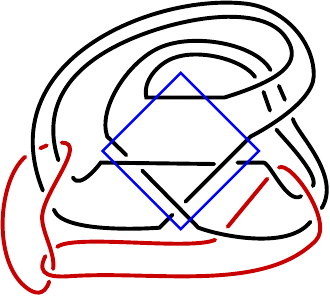} \hspace{0.5in}
  \includegraphics[height=1.1in]{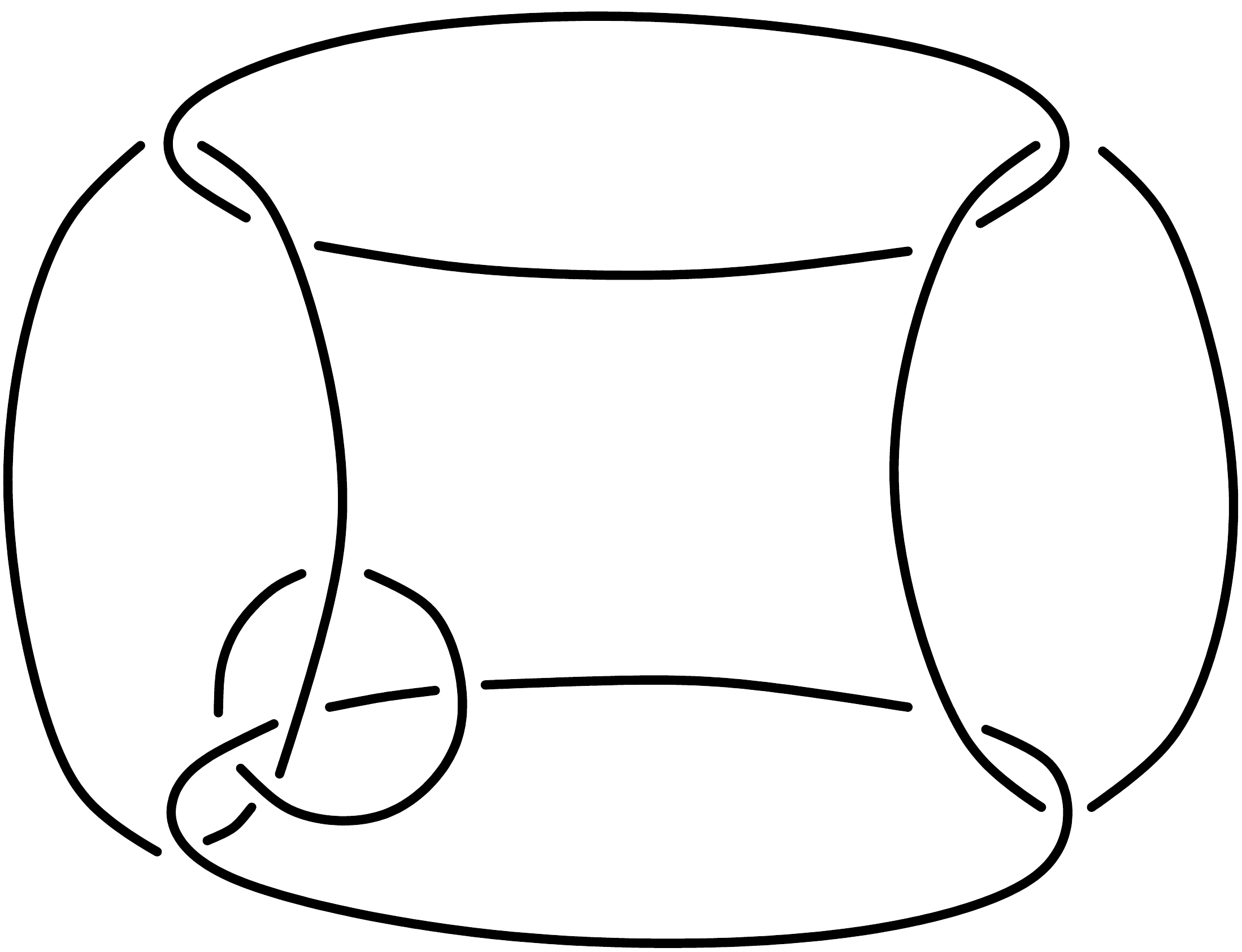} \hspace{0.5in}
  \includegraphics[height=1.2in]{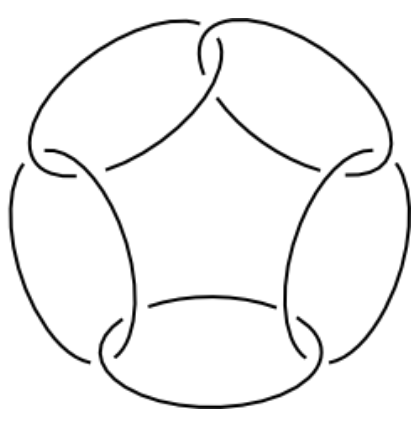} 
  \caption{$K$ (left and center) and minimally twisted $5$--chain link
    (right).}
  \label{Fig:5-chain}
\end{figure}

For a semi-regular biperiodic alternating link,
Theorem~\ref{Thm:BALvol} provides a decomposition of the link
complement into ideal hyperbolic torihedra. We say a torihedron is
\emph{right-angled} if it admits a hyperbolic structure in which all
dihedral angles on edges equal $\pi/2$. In this section we prove that
only two semi-regular biperiodic alternating links admit a
decomposition into right-angled torihedra: the square weave $\W$
studied in \cite{ckp:gmax}, and the triaxial link $\L$.

\begin{theorem}\label{Thm:square_triaxial}
The square weave $\W$ and the triaxial link $\L$ are the only
semi-regular links such that right-angled torihedra give the complete
hyperbolic structure. Thus, the links $\W$ and $\L$ have totally
geodesic checkerboard surfaces.
\end{theorem}

Theorem~\ref{Thm:square_triaxial} implies that no semi-regular
biperiodic alternating links have right-angled torihedra besides the
square weave and the triaxial link. It is still unknown whether there
are biperiodic alternating links that are not semi-regular with a
right-angled torihedral decomposition.

\begin{question}
Besides the square weave $\W$ and the triaxial link $\L$, do there
exist any other right-angled biperiodic alternating links?
\end{question}

\begin{proof}[Proof of Theorem~\ref{Thm:square_triaxial}]
First, we claim that if $\L$ has bigons then the torihedra cannot be
right-angled. For if there is a bigon, \reflem{BigonCollapse} implies
that there is an edge of degree six. If the torihedra were
right-angled, the angle sum on this edge would be $(\pi/2)*6 = 3\pi$,
which is impossible. Thus $\L$ has no bigons, and all vertices of
$T_{\L}$ are 4-valent.

By \refthm{BALvol}, there is a decomposition of any semi-regular
alternating link into regular ideal bipyramids on the faces of
$T_{\L}$. In this case, by Lemma~\ref{lemma:4valent}, the only faces
that occur are triangles, squares, and hexagons.

We obtain the torihedral decomposition from the bipyramid
decomposition of \refthm{BALvol} by first splitting each bipyramid
into two pyramids, and then gluing vertical faces. Splitting into
pyramids cuts in half the angle at the corresponding horizontal
edge. Recall that a horizontal edge of a regular ideal bipyramid over
an $n$--gon has angle $2\pi/n$ (see the proof of
\reflem{NoBigonsAngles}). Thus, splitting along a square gives angle
$\pi/4$, splitting along a hexagon gives angle $\pi/6$. We already
have two pyramids over triangular faces, which are regular ideal
tetrahedra with angle $\pi/3$.

Now, when we glue along vertical faces, the angle coming from one side
of the shared face is added to the angle coming from the other to give
the new angle on the horizontal edge. There are six cases:
\begin{enumerate}
\item A hexagon is adjacent to a hexagon. Then the angle on the
  adjacent horizontal edge is $\pi/6+\pi/6 = \pi/3$.
\item A hexagon is adjacent to a triangle. The angle is $\pi/6+\pi/3 =
  \pi/2$.
\item A hexagon is adjacent to a square. The angle is $\pi/6 + \pi/4 =
  5\pi/12$.
\item A triangle is adjacent to a triangle. The angle is $\pi/3+\pi/3
  = 2\pi/3$.
\item A triangle is adjacent to a square. The angle is $\pi/3 + \pi/4
  = 7\pi/12$.
\item A square is adjacent to a square. The angle is
  $\pi/4+\pi/4=\pi/2$.
\end{enumerate}

Figure~\ref{Fig:angles-alltypes} shows the angles on horizontal edges
of the torihedra for all five vertex types from
Lemma~\ref{lemma:4valent}. Note that the only cases that yield right
angles are the edges where two adjacent squares meet and the edges
where hexagons are adjacent to triangles. In order for the entire
torihedron to be right-angled, these are the only face adjacencies
possible. Then the only possible vertex types from
Lemma~\ref{lemma:4valent} are $3.6.3.6$ and $4.4.4.4$. These
semi-regular links are exactly the triaxial link and the square weave,
respectively.

Finally, note that if gluing geodesic ideal torihedra gives a complete hyperbolic structure, then the checkerboard surfaces inherit a pleating. That is, each surface is obtained by attaching totally geodesic polygonal faces, coming from the geodesic faces of the torihedra, joined along their edges at angles determined by the dihedral angles of the torihedra. 
The amount of bending, or pleating angle, is determined by the dihedral angles of the torihedra glued at that edge. In the case that the torihedra are right angled, then the pleating angle at each ideal edge will be $\pi/2+\pi/2=\pi$, or in other words the surface will be straight, not bent, at each edge. It follows that the surface is totally geodesic.
\end{proof}

\begin{figure}[h]
  \centering
  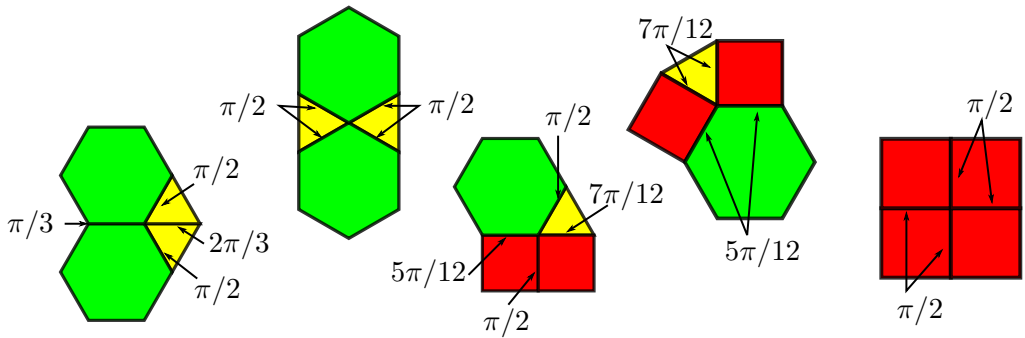
  \caption{Angles on horizontal edges for the five vertex types.}
  \label{Fig:angles-alltypes}
\end{figure}

%%%%%%%%%%%%%%%%%%%%%%%%%%%%%%%%%%%%%%%%%%%%%%%%%%%%%%%%%%%%%%%%%%%%%%%%%%%%%%%%%%%

\section{Proof of the Volume Density Conjecture for the triaxial link}
\label{sec:asymptotic}

In \cite{ck:det_mp, ckp:density, ckp:gmax}, we considered biperiodic
alternating links as limits of sequences of finite hyperbolic links.
We focused on the asymptotic behavior of two basic invariants, one
geometric and one diagrammatic, for a hyperbolic link $K$: The
\emph{volume density} of $K$ is defined as $\vol(K)/c(K)$, and the
\emph{determinant density} of $K$ is defined as
$2\pi\log\det(K)/c(K)$, where $c(K)$ denotes crossing number.  The
volume density is known to be bounded by the volume of the regular
ideal octahedron, $\voct \approx 3.66386$, and the same upper bound is
conjectured for the determinant density.

In \cite{ck:det_mp}, we defined the following notion of convergence of
links, and proved that for any sequence of alternating links $K_n$
that converge to a biperiodic alternating link $\L$ in this sense, the
determinant densities of $K_n$ converge to a type of determinant
density of $\L$.

\begin{definition}[\cite{ck:det_mp, ckp:gmax}]\label{def:folner_converge}
We will say that a sequence of alternating links $K_n$ {\em F{\o}lner
  converges almost everywhere} to the biperiodic alternating link
$\L$, denoted by $K_n\toF\L$, if the respective projection graphs
$\{G(K_n)\}$ and $G(\L)$ satisfy the following: There are subgraphs
$G_n\subset G(K_n)$ such that
\begin{enumerate}
\setlength\itemsep{0.5em}
\item[($i$)] $ G_n\subset G_{n+1}$, and $\bigcup G_n=G(\L)$,
\item[($ii$)] $\lim\limits_{n\to\infty}|\partial G_n|/|G_n|=0$, where
  $|\cdot|$ denotes number of vertices, and $\partial G_n\subset
  G(\L)$ consists of the vertices of $G_n$ that share an edge in
  $G(\L)$ with a vertex not in $G_n$,
\item[($iii$)] $ G_n\subset G(\L)\cap(n\Lambda)$, where $n\Lambda$
  represents $n^2$ copies of the $\Lambda$-fundamental domain for the
  lattice $\Lambda$ such that $L=\L/\Lambda$,
\item[($iv$)] $ \lim\limits_{n\to\infty} |G_n|/ c(K_n) = 1$.
\end{enumerate}
\end{definition}

\begin{definition}[Definition 4.4 \cite{ckp:gmax}]\label{def:cycle_of_tangles}
A diagram has \emph{no cycle of tangles} if whenever a disk embedded
in the torus meets the diagram transversely in exactly four edges,
then the disk contains a single twist region; i.e.\ a sequence of
bigons or exactly one crossing.
\end{definition}

As above, let $\W$ be the infinite square weave.  Using the
right-angled hyperbolic structure of $\W$ in an essential way, in
\cite{ckp:gmax} we proved:

\begin{theorem}\cite{ckp:gmax}\label{Thm:gmax}
Let $K_n$ be any alternating hyperbolic link diagrams with no cycles
of tangles such that $K_n\toF\W$.  Then for $K_n$, the volume and
determinant densities satisfy:
\[ \lim_{n\to\infty}\frac{\vol(K_n)}{c(K_n)}= \lim_{n\to\infty}\frac{2\pi\log\det(K_n)}{c(K_n)}=\voct. \]
\end{theorem}

The volume density of a biperiodic alternating link is defined as
\mbox{$\vol((T^2\times I)-L)/c(L)$}, where $c(L)$ is the crossing
number of the reduced alternating projection of $L$ on the torus,
which is minimal.  Hence, as $K_n\toF\W$, the volume densities of
$K_n$ converge to the volume density of $\W$, which is $\voct$; see
\cite{ckp:gmax}.

For determinant density, there is a toroidal invariant of $\W$ that
appears as the limit of the determinant density, namely the Mahler
measure of the two-variable characteristic polynomial of the toroidal
dimer model on an associated biperiodic graph, which measures the
entropy of the dimer model.  In \cite{ck:det_mp}, this diagrammatic
result for $\W$ was extended to {\em any} biperiodic alternating link
$\L$:

\begin{theorem}\cite{ck:det_mp}\label{Thm:det_mp}
Let $\L$ be any biperiodic alternating link, with alternating quotient
link $L$.  Let $p(z,w)$ be the characteristic polynomial of the
associated toroidal dimer model.  Then
\[ K_n\toF \L \quad \Longrightarrow \quad {\lim_{n\to\infty}\frac{\log\det(K_n)}{c(K_n)} = \frac{m(p(z,w))}{c(L)}}. \]
\end{theorem}

The following conjectures are motivated by Theorems~\ref{Thm:gmax} and
\ref{Thm:det_mp}:

\begin{conjecture}[Volume Density Conjecture]
\label{conj:vol_density}
Let $\L$ be any biperiodic alternating link, with alternating quotient
link $L$.  Let $K_n$ be alternating hyperbolic links such that
\mbox{$\displaystyle K_n\toF \L$}.  Then
\[
\lim_{n\to\infty}\frac{\vol(K_n)}{c(K_n)} = \frac{\vol((T^2\times I)-L)}{c(L)}
\]
\end{conjecture}

\begin{conjecture}[Toroidal Vol-Det Conjecture]
Let $\L$ be any biperiodic alternating link, with alternating quotient
link $L$.  Let $K_n$ be alternating hyperbolic links such that
\mbox{$\displaystyle K_n\toF \L$}.  Then
\[ \vol((T^2\times I)-L)  \leq 2\pi\, m(p(z,w)).\]
\end{conjecture}

The Toroidal Vol-Det Conjecture is studied in more detail in
forthcoming work \cite{ckl:bal_mm}.

Theorem~\ref{Thm:gmax} proves both conjectures with equality for the
square weave $\W$.  In Theorem~\ref{Thm:triaxial}, we establish both
conjectures with equality for the triaxial link as well.  By
Theorem~\ref{Thm:square_triaxial}, these are the only semi-regular
links such that right-angled torihedra give the complete hyperbolic
structure.

\begin{theorem}\label{Thm:triaxial}
Let $\L$ be the triaxial link.  Let $K_n$ be any alternating
hyperbolic link diagrams with no cycles of tangles such that
$K_n\toF\L$.  Then
\[ \lim_{n\to\infty} \frac{\vol(K_n)}{c(K_n)} =
\lim_{n\to\infty} \frac{2\pi\log\det(K_n)}{c(K_n)} =
\frac{\vol((T^2\times I)-L)}{c(L)} = \frac{10\,\vtet}{3}. \]
\end{theorem}

\begin{proof}
By Theorem~\ref{Thm:BALvol},
\[ \frac{\vol((T^2\times I)-L)}{c(L)}=\frac{10\,\vtet}{3}.\]
By \cite[Example 4.2]{ck:det_mp},
\[ \lim_{n\to\infty}\frac{2\pi\log\det(K_n)}{c(K_n)} = \frac{2\pi\, m(p(z,w))}{c(L)} =\frac{10\,\vtet}{3}.\]
It remains to prove Conjecture~\ref{conj:vol_density} for the triaxial
link $\L$.  Namely,
\[ K_n \toF \L \implies {\lim_{n\to\infty}\frac{\vol(K_n)}{c(K_n)} = \frac{\vol((T^2\times I)-L)}{c(L)}}. \]

We adapt the proof of \cite[Theorem 1.4]{ckp:gmax}, which proves
Conjecture~\ref{conj:vol_density} for $\W$.

By Theorem~\ref{Thm:square_triaxial}, the checkerboard surfaces of
$\L$ are totally geodesic, with the faces of the torihedra lifting to
totally geodesic hyperplanes in the universal cover $\HH^3$ of
$\RR^3-\L$, meeting at right angles. This forces these hyperplanes to
meet $\partial \HH^3 = \widehat\CC$ in the circle pattern shown in
\reffig{triaxial_circles}.

\begin{figure}
  \centering 
  \includegraphics{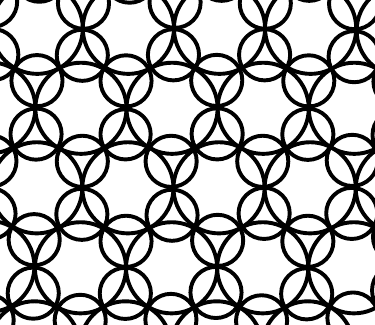}
  \caption{Circle pattern for $P(\L)$ for the triaxial link $\L$.  }
  \label{Fig:triaxial_circles}
\end{figure}

Now proceed as in \cite{ckp:gmax}. If $K_n\subset S^3$ is a sequence
of alternating hyperbolic links such that $K_n\toF\L$, then by
\cite[Theorem~4.13]{ckp:gmax} the volumes of $K_n$ are bounded below
by twice the volume of the polyhedron obtained from the checkerboard
polyhedron of $K_n$ by assigning right angles to all the edges. Denote
this by $P_n$.

Now we repeat the proof of \cite[Lemma~5.3]{ckp:gmax}, replacing the
disk pattern coming from $\W$ with that of \reffig{triaxial_circles}
coming from $\L$. In that proof, each instance of $\voct$ (the volume
density of $\W$) must be replaced by $(10\vtet/3)$, the volume density
of $\L$.

Finally, step through the proof of Theorem~1.4 in
\cite[section~5]{ckp:gmax}. Again, first check that the hypotheses are
satisfied for the lemma analogous to \cite[Lemma~5.3]{ckp:gmax} (with
$\voct$ replaced by $10\vtet/3$). Note for condition (2), the argument
is the same but the constant $4$ is replaced by $6$ (this does not
affect the limit). For (1), $|\partial B(x,\ell)| = 6\ell$. Otherwise,
the argument is identical. Then this lemma and
\cite[Theorem~4.13]{ckp:gmax} imply
\[ \lim_{n\to \infty} \frac{\vol(K_n)}{c(K_n)} = \frac{10\vtet}{3}.\qedhere \]
\end{proof}

%%%%%%%%%%%%%%%%%%%%%%%%%%%%%%%%%%%%%%%%%%%%%%%%%%%%%%%%%%%%%%%%%

\section{Hyperbolicity for alternating links in $T^2\times I$}
\label{sec:hyperbolic}

In this section, we generalize the results on hyperbolicity and
triangulations to wider classes of links on $T^2\times I$. We prove
that if a link $L$ in $T^2\times I$ admits a certain kind of
alternating diagram on the torus, then $(T^2\times I)-L$ is
hyperbolic. Moreover, for a (possibly incomplete) hyperbolic structure
on the complement, we find a triangulation that satisfies Thurston's
gluing equations.  We show that this triangulation is positively
oriented and unimodular.  This leads to lower and upper bounds on the
hyperbolic volume of $(T^2\times I)-L$, denoted by $\vol(L)$.

To describe the links, we need a few definitions. First, the
definition of connected sums of knots in $S^3$ can be extended to
knots in $T^2\times I$. We are concerned with knots with reduced
diagrams that are not connected sums, as in the following definition.

\begin{definition}\label{Def:WeaklyPrime}
A diagram is \emph{weakly prime} if whenever a disk embedded in the
diagram surface meets the diagram transversely in exactly two edges,
then the disk contains a simple edge of the diagram and no crossings.
\end{definition}

In addition, we will also need Definition~\ref{def:cycle_of_tangles} and the following:

\begin{definition}\label{def:unimodular}
A geodesic ideal triangulation satisfying the edge gluing equations is
called {\em unimodular} if every tetrahedron has an edge parameter $z$
that satisfies $|z|=1$.
\end{definition}

By Equation~(\ref{eqn:edge-param}) below, a tetrahedron is unimodular if and only if its vertex link triangles are isosceles.

\begin{definition}\label{def:volbp}
Let $B_n$ denote the hyperbolic regular ideal bipyramid whose link
polygons at the two coning vertices are regular $n$--gons. For a face
$f$ of $G(L)$, let $|f|$ denote the degree of the face. Let $L$ be a
link in $T^2\times I$ with an alternating diagram on $T^2\times\{0\}$,
and let $T_L$ be the toroidal graph defined in
Lemma~\ref{Lem:BigonCollapse}.  Define the \emph{bipyramid volume} of
$L$ as
\[ \volbp(L) = \sum_{f \in \{\text{faces of} \ T_L\}}\vol(B_{|f|}) . \]
\end{definition}
Note that for all semi-regular links, Theorem \ref{Thm:BALvol} implies
$\vol(L)= \volbp(L)$.

\begin{definition}\label{def:volp}
Let $L$ be a link in $T^2\times I$ with an alternating diagram on
$T^2\times \{0\}$, and let $P(L)$ be the torihedron as in Theorem
\ref{Thm:Torihedra}. Assume $P(L)$ admits an ideal right-angled
hyperbolic structure.
Define the \emph{right-angled volume} of $L$ as
\[ \volp(L) = 2 \vol(P(L)). \]
\end{definition}
Note that for the square weave and triaxial link,
Theorem~\ref{Thm:square_triaxial} implies $\vol(L)= \volp(L)$.

Suppose $K$ is a link in $S^3$ with a prime, alternating,
twist-reduced diagram with no cycle of tangles, and with bigons
removed.  In \cite[Theorem~4.13]{ckp:gmax}, we proved that the two
checkerboard polyhedra coming from $K$, when given an ideal hyperbolic
structure with all right angles, have volume providing a lower bound
for $\vol(S^3-K)$.

In Theorem~\ref{Thm:hyperbolic} below, we extend
\cite[Theorem~4.13]{ckp:gmax} to any link in $T^2\times I$ satisfying
similar conditions.  The methods involved in proving
Theorem~\ref{Thm:hyperbolic} are different from those used in
\cite{ckp:gmax}, which relied on volume bounds via guts of
3--manifolds cut along essential surfaces.  Here, the proof of
Theorem~\ref{Thm:hyperbolic} involves circle packings, geometric
structures on triangulations, and the convexity of volume.

\begin{theorem}\label{Thm:hyperbolic}
Let $L$ be a link in $T^2\times I$ with a weakly prime alternating
diagram on $T^2\times \{0\}$ with no bigons.  If $L$ has no cycle of
tangles, then
\begin{enumerate}
\item\label{Itm:Hyp} $(T^2\times I)-L$ is hyperbolic.
\item\label{Itm:Trian} $(T^2\times I)-L$ admits a (possibly
  incomplete) hyperbolic structure obtained from an ideal, positively
  oriented, unimodular triangulation.
\item\label{Itm:RightAngled} Under the structure of
  item~\eqref{Itm:Trian}, the torihedron $P(L)$ is right-angled.
\item\label{Itm:Vol} $\volp(L) \leq \vol(L) \leq \volbp(L)$.
\end{enumerate}
\end{theorem}

Note that by Theorems \ref{Thm:BALvol} and \ref{Thm:square_triaxial},
both inequalities in \eqref{Itm:Vol} become equalities for the square
weave and the triaxial link: all three give the volume of the complete
structure. Thus both upper and lower bounds of
Theorem~\ref{Thm:hyperbolic}, \eqref{Itm:Vol} are sharp.  In general,
the bounds in \eqref{Itm:Vol} are volumes of incomplete ideal
hyperbolic structures on $(T^2\times I)-L$.

\begin{remark}
Theorem~\ref{Thm:hyperbolic} should be compared to the results of
\cite{HowiePurcell}. That paper also implies that the links in
\refthm{hyperbolic} are hyperbolic, and gives a lower bound on volume
in terms of the number of twist regions of the diagram. Since the
diagram here has no bigons, this also amounts to a lower volume bound
in terms of the crossing number. However, the results of
\refthm{hyperbolic} are stronger in this case because they give an
explicit hyperbolic structure, although that structure is likely
incomplete. Moreover, the volume bound of \refthm{hyperbolic} is known
to be sharp for the square weave and triaxial link.
\end{remark}

The proof of \refthm{hyperbolic} will proceed in the order
\eqref{Itm:RightAngled}, \eqref{Itm:Trian}, \eqref{Itm:Hyp},
\eqref{Itm:Vol}. The proof relies on a fundamental result about the
existence of certain circle patterns on the torus, due to Bobenko and
Springborn \cite{bs2004}.
The following special case of \cite[Theorem 4]{bs2004} applies.

\begin{theorem}[\cite{bs2004}]\label{Thm:bobenko-springborn}
Suppose $G$ is a 4-valent graph on the torus $T^2$, and $\theta\in
(0,2\pi)^E$ is a function on edges of $G$ that sums to $2\pi$ around
each vertex. Let $G^*$ denote the dual graph of $G$. Then there exists
a circle pattern on $T^2$ with circles circumscribing faces of $G$
(after isotopy of $G$) and having exterior intersection angles
$\theta$, if and only if the following condition is satisfied:
\begin{itemize}
\item[] Suppose we cut the torus along a subset of edges of
  $G^*$, obtaining one or more pieces. For any piece that is a disk,
  the sum of $\theta$ over the edges in its boundary must be at least
  $2\pi$, with equality if and only if the piece consists of only one
  face of $G^*$ (only one vertex of $G$).
\end{itemize}
The circle pattern on the torus is uniquely determined up to
similarity.
\end{theorem}

\begin{proof}[Proof of \refthm{hyperbolic}]
For $L$ as stated, let $G(L)$ be its projection graph on
$T^2\times\{0\}$, and let $\theta(e)=\pi/2$ for every edge $e$ in
$G(L)$.  For this choice of angles, we now verify the condition of
Theorem~\ref{Thm:bobenko-springborn}.  This will prove the existence
of an orthogonally intersecting circle pattern that is combinatorially
equivalent to $G(L)$.

Let $C$ be a simple closed curve in $T^2\times \{0\}$ that intersects
$n$ edges of $G(L)$ transversely, and bounds a disk in
$T^2\times\{0\}$.  Since $G(L)$ is weakly prime, $n\geq 3$. If $n\geq
5$, the angle sum is greater than $2\pi$. Hence, it remains only to
check the cases $n=3$ and $n=4$.

The curve $C$ bounds a disk $D$. Let $V_I$ denote the number of
vertices of $G(L)$ that lie in $D$, and let $E_I$ denote the number of
edges of $G(L)$ inside $D$ disjoint from $C$. Because $G(L)$ is
4-valent, $n+2E_I = 4V_I$ . It follows that $n$ is even, ruling out
$n=3$.

So suppose $n=4$. In this case, the angle sum equals $2\pi$. Hence,
the condition of Theorem~\ref{Thm:bobenko-springborn} is satisfied if
and only if the disk bounded by $C$ contains exactly one crossing.
Since $G(L)$ has no bigons, this condition is equivalent to the
condition that $L$ has no cycle of tangles, which holds by
hypothesis. Thus \refthm{bobenko-springborn} applies.

Therefore, there exists an orthogonal circle pattern on the torus with
circles circumscribing the faces of $G(L)$.  Lifting this circle
pattern to the universal cover of the torus defines an orthogonal
biperiodic circle pattern on the plane. Considered as the plane at
infinity for ${\mathbb H}^3$, this circle pattern defines a
right-angled biperiodic ideal hyperbolic polyhedron in $\HH^3$. The
torihedron $P(L)$ is its toroidal quotient, which is realized as an
ideal right-angled hyperbolic torihedron, as required for
Theorem~\ref{Thm:hyperbolic} part \eqref{Itm:RightAngled}.

\begin{figure}
\begin{center}
\begin{tabular}{ccc}
  \includegraphics{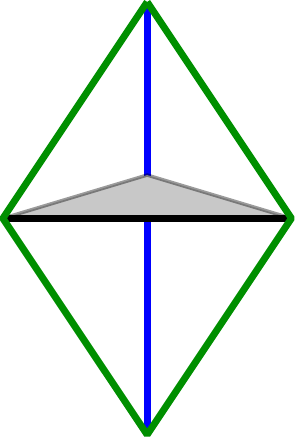} &
  \includegraphics{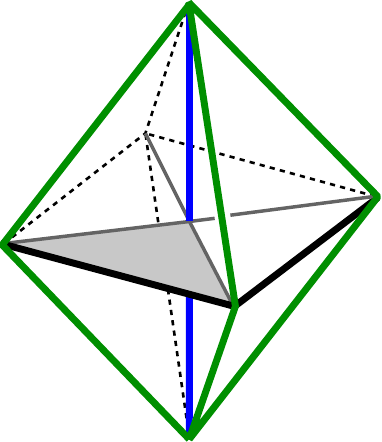} &
  \includegraphics{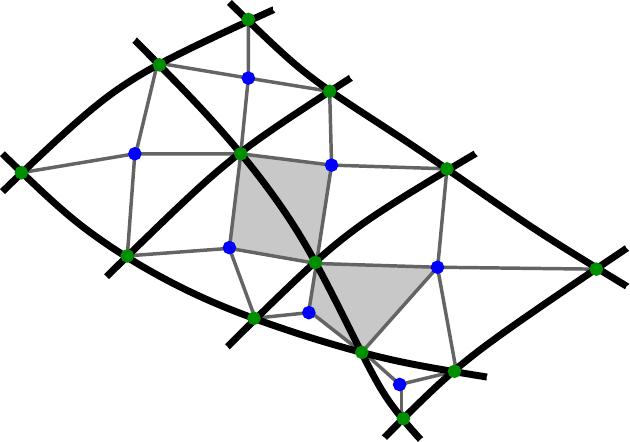}\\
  & & \\
  (a) & (b) & (c) \\
\end{tabular}
\end{center}
\caption{(a) A tetrahedron in $\T$ with stellating (blue), vertical (green), and horizontal (black) edges.
  A link triangle for an ideal vertex at $\infty$ is shaded.
  (b) Tetrahedra are glued at the stellating edge, centrally triangulating faces of $G(L)$ on $T^2\times\{0\}$.
  (c) Graphs $G(L)$ (black) and $G^c_L$ (gray and black).
  Shaded triangles indicate the four tetrahedra glued along one horizontal edge.}
\label{Fig:quad-graph}
\end{figure}

%%%%%%%%
Now, let $\T$ be the stellated bipyramid triangulation. The
right-angled structure can be used to assign shape parameters in $\CC$
to the tetrahedra of $\T$.  In particular, let $G^c_L$ be the central
triangulation of $G(L)$, given by adding a central vertex in each face
and edges in each face to triangulate the faces; see
\reffig{quad-graph}.  Since the condition of
Theorem~\ref{Thm:bobenko-springborn} holds, each face of $G(L)$ is
inscribed in a geometric circle, such that the center of the circle
lies in the interior of the face.  Thus, we may take the center of the
circle to be the central vertex of $G_L^c$.  Therefore, every triangle
of $G_L^c$ is isosceles, with the two equal edges coming from $G^c_L -
E(G(L))$.  When the isosceles triangles are paired, they form
right-angled kites, which is the quad graph of the orthogonal circle
pattern.

Each triangle of $G^c_L$ corresponds to a unique tetrahedron in $\T$.
We assign edge parameters to a tetrahedron of $\T$ as follows. Recall
that for an ideal hyperbolic tetrahedron, a \emph{vertex link
  triangle} is the boundary of a horoball neighborhood of an ideal
vertex. We view the triangles of $G^c_L$ as vertex link triangles for
$\T$. If the angles of the link triangle are $\alpha$, $\beta$,
$\gamma$ in clockwise order, then the edge parameter of the edge with
dihedral angle $\alpha$ is given by
\begin{equation}\label{eqn:edge-param}
z(\alpha)=e^{i\alpha} \sin(\gamma)/\sin(\beta).
\end{equation}
This assigns edge parameters to tetrahedra in $\T$. 

Since every angle of every triangle is in $(0,\pi)$, every edge
parameter has positive imaginary part; i.e.\ all the tetrahedra are
positively oriented. Moreover, because every triangular face of
$G^c_L$ is isosceles, with the congruent angles adjacent to the edges
of $G(L)$, equation (\ref{eqn:edge-param}) implies that the edge
parameter at the stellating edge of every tetrahedron in $\T$ equals
$e^{i\alpha}$ for some $\alpha$. Since every tetrahedron has an edge
parameter which lies on the unit circle, the ideal triangulation is
unimodular.

To show the ideal tetrahedra with these edge parameters give a
(possibly incomplete) hyperbolic structure on $(T^2\times I)-L$, we
will show that the edge gluing equations are satisfied for each of the
three types of edges (stellating, vertical and horizontal).

First, note that the angle sum around a stellating edge is $2\pi$
because of the planar embedding of $G^c_L$. If the angles are
$\alpha_1, \dots, \alpha_k$, then the corresponding edge gluing
equation is $e^{i\alpha_1}\dots e^{i\alpha_k} = 1$. Hence, the edge
gluing equations are satisfied at the stellating edges.

At a vertical edge, the planar embedding again implies that the angle
sum is $2\pi$, satisfying the rotational part of the edge gluing
equation. Because the triangles of the graph $G^c_L$ around a vertex
of $G(L)$ are exactly the vertex link triangles of the ideal
tetrahedra adjacent to the vertical edge, the tetrahedra glue without
translation along the edge, hence rotational and translational parts
of the equation are both satisfied.

Since $L$ has no bigons, the horizontal edges of $\T$ are identified
in pairs that share a common vertex on the torihedron $P(L)$ (see
Figure \ref{Fig:4valent-polyhedron}).  Hence, there are four
tetrahedra whose horizontal edges are identified to one edge in $\T$.
For each tetrahedron in $\T$, each horizontal edge is opposite to the
stellating edge, so it has the same edge parameter $e^{i\alpha}$ and
thus is unimodular.

For the four tetrahedra whose horizontal edges are identified, their
link triangles glue along the horizontal edges to form two
right-angled kites, whose diagonals are these edges. See
Figure~\ref{Fig:quad-graph}.  The dihedral angle at the horizontal
edge equals the corresponding vertex angle of the kite.  Since every
kite is right-angled, its non-congruent vertex angles sum to $\pi$.
So for each horizontal edge, the angle contribution from these four
tetrahedra is $2\pi$.

Now, because the horizontal edge parameters are unimodular, these edge
parameters multiply to $1$.  Therefore, the edge parameters at the
horizontal edges satisfy the edge gluing equations. This completes the
proofs of Theorem~\ref{Thm:hyperbolic} parts~\eqref{Itm:Trian}
and~\eqref{Itm:RightAngled}.

A positively oriented geometric triangulation implies
part~\eqref{Itm:Hyp}; see for example~\cite{choi}.

To prove the lower bound, we use the lower volume bound given by the
volume of a representation.  Because $(T^2\times I)-L$ admits a
hyperbolic structure, it admits a representation of the fundamental
group into ${\rm PSL}(2,\CC)$.  The volume of this representation is
$2\;\vol(P(L))$.  Francaviglia showed that the volume of such a
representation gives a lower bound on the volume of the complete
structure~\cite{francaviglia2}. The lower bound in
part~\eqref{Itm:Vol} follows.

Finally, the tetrahedra in $\T$ combine along every stellating edge to
give a hyperbolic ideal bipyramid over a face of $T_L$.  The upper
bound in part \eqref{Itm:Vol} follows from the fact that the maximal
volume of a complete ideal $n$--bipyramid is obtained uniquely by the
regular $n$--bipyramid \cite[Theorem~2.1]{Adams:BipyramidVolume}.

This completes the proof of Theorem~\ref{Thm:hyperbolic}.
\end{proof}

\bibliographystyle{amsplain} 
\bibliography{references}
\end{document}

%% file: figures/4valent-polyhedron.pdf_tex
\begin{tikzpicture} [every node/.style={scale=1}]
\node at (0,0)
{\includegraphics[width=3in]{figures/4valent-polyhedron.pdf}};

\node at (-1.75,0) {edge $e$};

\node at (4,1) { $e$};
\node at (2,3) { $e$};

\node at (4,-0.25) { $e$};
\node at (1,-3) { $e$};

\end{tikzpicture}

%% file: figures/bigons.pdf_tex
\begin{tikzpicture}
\node at (0,0)
{  \includegraphics[width=6in]{figures/bigons.pdf}};
 \node at (-3.85,1){$e$};
 \node at (-3,1){$f$};
\node at (-.5,1.5){$e$};
\node at (1,.25){$e$};
\node at (1,1.85){$f$};
\node at (2.5,.25){$f$};
\node at (5,0.25) {edges $e = f$};

 \node at (-6.75,-1){$e$};
 \node at (-5.75,-1){$f$};
 \node at (-4.65,-1){$f$};
 \node at (-3.75,-1){$g$};

\node at (-1.5,-.5){$e$};
\node at (0,-1.75){$e$};
\node at (0,-.15){$f$};
\node at (2,-1.85){$f$};
\node at (2,-.15){$g$};
\node at (3.5,-1.75){$g$};

\node at (6,-1.75) {edges $e = f = g$};

%\draw[help lines] (-8,-2) grid (8,2);
%\node at (0,0){OO};
% \node at (-6,-2.25) {(a)};
% \node at (-1.5,-2.25) {(b)};
% \node at (3,-2.25) {(c)};
% \node at (7,-2.25) {(d)};
\end{tikzpicture}

%% file: figures/hex-face.pdf_tex
\begin{tikzpicture}
%\draw[help lines] (-8,-4) grid (8,4);
%\node at (0,4){OO};

\node at (-7.5,0)
%{\includegraphics[width=.65in]{figures/hex-face}};
{\includegraphics{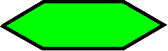}};
\node at (-7.5,-1){Face of $T_{\L}$};
\node at (-4,1.2)
%{\includegraphics[width=.65in]{figures/hex-pyramid.pdf}};
{\includegraphics{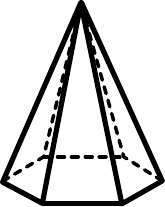}};
\node at (-4,-1.2)
%{\includegraphics[width=.65in, angle=180]{figures/hex-pyramid-1.pdf}};
{\includegraphics[angle=180]{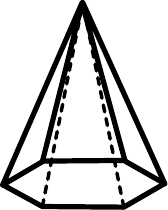}};
\node at (-4,-3){Pyramids};
\node at (0,0)
%{\includegraphics[width=1in]{figures/hex-bipyramid.pdf}}; 
{\includegraphics{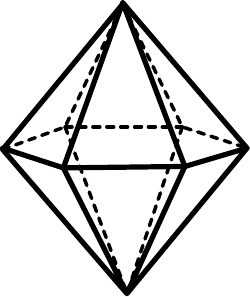}}; 
\node at (0,-2){Bipyramid formed by};
\node at (0,-2.5){gluing two pyramids };
\node at (4.5,0)
%{\includegraphics[width=1in]{figures/hex-bipyramid.pdf}}; 
{\includegraphics{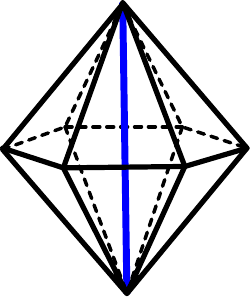}}; 
%% \draw[blue,thick](4.5,1.45) -- (4.5,-1.45);
\node at (4.5,-2){Stellating edge};
\draw[-] (-5.95,.75) -- (-5.95,.25);
\draw[->] (-5.95,.25) -- (-4.8,.25);
\node at (-6.25,1){horizontal edge};
%\node[scale=0.65] at (-7,1){horizontal edges};
\draw[->] (-5.7,2) -- (-4.45,1.35);
\node at (-6,2.25){vertical edge};
%\node[scale=0.65] at (-7,2){vertical edges};
 \end{tikzpicture}

%% file: figures/4valent-anglesum.pdf_tex
\begin{tikzpicture}
\node at (0,0)
{\includegraphics[width=5in]{figures/4valent-anglesum.pdf}};
%\draw[help lines] (-8,-2) grid (8,2);
%\node at (0,0){OO};
\node at (-6.25,0.5){$n_4$--gon};
\node at (-4.75,1.5){$n_3$--gon};
\node at (-2.5,0.5){$n_2$--gon};
\node at (-4.75,-1.5){$n_1$--gon};
\node at (-4.5,-1){$\theta_1$};
\node at (-3.5,0){$\theta_2$};
\node at (-4.5,1){$\theta_3$};
\node at (-5.4,0){$\theta_4$};
\node at (4,1) { $e$};
\node at (5.75,-.75) { $e$};
\node at (4.8,-1.5){$\alpha_1$};
\node at (4,1.5){$\alpha_3$};
\node at (6,-0.35){$\alpha_2$};
\node at (3,0.3){$\alpha_4$};
\node at (-4.5,-2.3) {(a)};
\node at (4.5,-2.3) {(b)};
\end{tikzpicture}

%% file: figures/3valent-anglesum.pdf_tex
\begin{tikzpicture}
\node at (0,0)
{\includegraphics[width=5in]{figures/3valent-anglesum.pdf}};
%\draw[help lines] (-8,-2) grid (8,2);
%\node at (0,0){OO};

\node at (3.8,-1.1){$\alpha_1$};
\node at (4.4,-1.4){$\alpha_1$};
\node at (3.4,1.2){$\alpha_3$};
\node at (2.8,1.3){$\alpha_3$};
\node at (6,-0.3){$\alpha_2$};
\node at (1.25,0.3){$\alpha_4$};

\node at (-6.5,0){$n_4$--gon};
\node at (-3.8,1.25){$n_3$--gon};
\node at (-1.1,0){$n_2$--gon};
\node at (-3.8,-1.25){$n_1$--gon};

\node at (-2.9,-0.5){$\theta_1$};
\node at (-2,0){$\theta_2$};
\node at (-2.9,0.65){$\theta_3$};

\node at (-4.6,-0.5){$\theta_1$};
\node at (-4.6,0.65){$\theta_3$};
\node at (-5.5,0){$\theta_4$};

\node at (-3.75,0.25){edge $e$};
\node at (2.2,0.75){ $e$};
\node at (3.45,0.35){ $e$};
\node at (5,-0.75){ $e$};

\node at (-4,-2.25) {(a)};
\node at (4,-2.25) {(b)};

\end{tikzpicture}

%% file: figures/Angles-Alltypes.pdf_tex
%% Creator: Inkscape inkscape 0.92.1, www.inkscape.org
%% Accompanies image file 'Angles-Alltypes.pdf' (pdf, eps, ps)
\begingroup%
  \makeatletter%
  \providecommand\color[2][]{%
    \errmessage{(Inkscape) Color is used for the text in Inkscape, but the package 'color.sty' is not loaded}%
    \renewcommand\color[2][]{}%
  }%
  \providecommand\transparent[1]{%
    \errmessage{(Inkscape) Transparency is used (non-zero) for the text in Inkscape, but the package 'transparent.sty' is not loaded}%
    \renewcommand\transparent[1]{}%
  }%
  \providecommand\rotatebox[2]{#2}%
  \ifx\svgwidth\undefined%
    \setlength{\unitlength}{394.91619873bp}%
    \ifx\svgscale\undefined%
      \relax%
    \else%
      \setlength{\unitlength}{\unitlength * \real{\svgscale}}%
    \fi%
  \else%
    \setlength{\unitlength}{\svgwidth}%
  \fi%
  \global\let\svgwidth\undefined%
  \global\let\svgscale\undefined%
  \makeatother%
  \begin{picture}(1,0.31510187)%
    \put(0,0){\includegraphics[width=\unitlength,page=1]{Angles-Alltypes.pdf}}%
    \put(0.91353737,0.21337621){\color[rgb]{0,0,0}\makebox(0,0)[lb]{\smash{$\pi/2$}}}%
    \put(0.84676101,0.0165142){\color[rgb]{0,0,0}\makebox(0,0)[lb]{\smash{$\pi/2$}}}%
    \put(0,0){\includegraphics[width=\unitlength,page=2]{Angles-Alltypes.pdf}}%
    \put(0.5997903,0.28397032){\color[rgb]{0,0,0}\makebox(0,0)[lb]{\smash{$7\pi/12$}}}%
    \put(0.6828386,0.07175731){\color[rgb]{0,0,0}\makebox(0,0)[lb]{\smash{$5\pi/12$}}}%
    \put(0,0){\includegraphics[width=\unitlength,page=3]{Angles-Alltypes.pdf}}%
    \put(0.45226127,0.00526713){\color[rgb]{0,0,0}\makebox(0,0)[lb]{\smash{$\pi/2$}}}%
    \put(0.36414959,0.05167621){\color[rgb]{0,0,0}\makebox(0,0)[lb]{\smash{$5\pi/12$}}}%
    \put(0.55270135,0.12651321){\color[rgb]{0,0,0}\makebox(0,0)[lb]{\smash{$7\pi/12$}}}%
    \put(0.51141487,0.20132989){\color[rgb]{0,0,0}\makebox(0,0)[lb]{\smash{$\pi/2$}}}%
    \put(0,0){\includegraphics[width=\unitlength,page=4]{Angles-Alltypes.pdf}}%
    \put(0.39952457,0.21251987){\color[rgb]{0,0,0}\makebox(0,0)[lb]{\smash{$\pi/2$}}}%
    \put(0.20050465,0.21072984){\color[rgb]{0,0,0}\makebox(0,0)[lb]{\smash{$\pi/2$}}}%
    \put(0,0){\includegraphics[width=\unitlength,page=5]{Angles-Alltypes.pdf}}%
    \put(0.19035609,0.0792518){\color[rgb]{0,0,0}\makebox(0,0)[lb]{\smash{$2\pi/3$}}}%
    \put(0.17533427,0.15081959){\color[rgb]{0,0,0}\makebox(0,0)[lb]{\smash{$\pi/2$}}}%
    \put(0.17642401,0.03597634){\color[rgb]{0,0,0}\makebox(0,0)[lb]{\smash{$\pi/2$}}}%
    \put(-0.00197054,0.0984766){\color[rgb]{0,0,0}\makebox(0,0)[lb]{\smash{$\pi/3$}}}%
    \put(0,0){\includegraphics[width=\unitlength,page=6]{Angles-Alltypes.pdf}}%
  \end{picture}%
\endgroup%